\documentclass[11pt]{amsart}
\usepackage{color,array,amssymb,amsmath, amsthm, amsfonts, tikz-cd}

\usepackage[notcite, notref, final]{showkeys}

\usepackage{graphicx}
\usepackage{listings}
\usepackage[margin=1in]{geometry}
\usepackage{lstautogobble}
\usepackage{enumerate}
\usepackage{thmtools}
\usepackage{thm-restate}
\usepackage{amsthm}
\usepackage{verbatim}

\usepackage{mathtools}
\usepackage{physics}
\usepackage{color}
\usepackage{hyperref}
\usepackage[capitalise]{cleveref}
\usepackage[bottom]{footmisc}
\crefformat{equation}{(#2#1#3)}

\usepackage{float}
\restylefloat{table}

\usepackage{mathrsfs}

\theoremstyle{plain}
\newtheorem{thm}{Theorem}[section]
\newtheorem{prop}[thm]{Proposition}
\newtheorem{lemma}[thm]{Lemma}

\theoremstyle{definition}

\newtheorem{remark}[thm]{Remark}

\theoremstyle{plain}
\newtheorem{mthm}{Theorem}

\Crefname{prop}{Proposition}{Propositions}
\Crefname{mthm}{Theorem}{Theorems}

\numberwithin{equation}{section} 

\newcommand{\M}{{\mathcal{M}}}

\newcommand{\p}{{\partial}}

\renewcommand{\d}{\mathsf{d}}

\newcommand{\R}{{\mathbb{R}}}

\newcommand{\Q}{{\mathbb{Q}}}

\newcommand{\g}{{\mathsf{g}}}

\renewcommand{\L}{{\mathcal{L}}}

\renewcommand{\M}{{\mathbb{M}}}

\renewcommand{\k}{\mathsf{k}}
\newcommand{\I}{\mathbb{I}}

\newcommand{\tl}{\tilde}

\newcommand{\D}{\Delta}

\newcommand{\ph}{\phantom{=}}
\newcommand{\nn}{\nonumber}

\newcommand{\al}{\alpha}
\newcommand{\be}{\beta}
\newcommand{\ka}{\kappa}
\newcommand{\la}{\lambda}

\newcommand{\K}{\mathsf{K}}
\newcommand{\XN}{X_N}

\newcommand{\E}{{\mathbb{E}}}

\newcommand{\Cs}{\mathsf{C}}
\newcommand{\Fr}{{F}}
\newcommand{\Ec}{\mathcal{E}}
\renewcommand{\P}{\mathcal{P}}

\newcommand{\bmu}{\bar\mu}

\newcommand{\ga}{\gamma}

\renewcommand{\sf}{\mathsf{s}}

\newcommand{\bFr}{\bar{F}_N}
\newcommand{\bEc}{\bar{\Ec}_N}
\newcommand{\ws}{\mathsf{w}}
\newcommand{\os}{\mathsf{o}}

\let\div\relax
\DeclareMathOperator{\div}{div}

\let\PV\relax
\DeclareMathOperator{\PV}{PV}

\def\XXint#1#2#3{{\setbox0=\hbox{$#1{#2#3}{\int}$ }
\vcenter{\hbox{$#2#3$ }}\kern-.6\wd0}}

\def \hal{\frac{1}{2}}
\def\({\left(}
\def\){\right)}

\def\nab{\nabla}
\def\indic{\mathbf{1}}

\title[Relative entropy and modulated free energy without confinement]{Relative entropy and modulated free energy without confinement via self-similar transformation}
\author[M. Rosenzweig]{Matthew Rosenzweig}
\address{Matthew Rosenzweig, Carnegie Mellon University, Department of Mathematical Sciences, Pittsburgh, PA} 
\email{mrosenz2@andrew.cmu.edu}
\thanks{~M.R. is supported by NSF grant DMS-2206085.}
\author[S. Serfaty]{Sylvia Serfaty}
\address{Sylvia Serfaty, Courant Institute of Mathematical Sciences, New York University, New York City, NY}
\thanks{~S.S. is supported by NSF grant DMS-2247846 and a Simons Investigator award.}
\email{serfaty@cims.nyu.edu}
\begin{document}
\begin{abstract}
This note extends the modulated entropy and free energy methods for proving mean-field limits/propagation of chaos to the whole space without any confining potential, in contrast to previous work limited to the torus or requiring confinement in the whole space, for all log/Riesz flows.  Our novel idea is a scale transformation, sometimes called self-similar coordinates in the PDE literature, which converts the problem to one with a quadratic confining potential, up to a time-dependent renormalization of the interaction potential. In these self-similar coordinates, one can then establish a Gr\"onwall relation for the relative entropy or modulated free energy, conditional on bounds for the Hessian of the mean-field log density. This generalizes recent work of Feng-Wang \cite{FW2023}, which extended the relative entropy method of \cite{JW2018} to the whole space for the viscous vortex model. Moreover, in contrast to previous work, our approach allows to obtain uniform-in-time propagation of chaos and even polynomial-in-time generation of chaos in the whole space without confinement, provided one has suitable decay estimates for the mean-field log density. The desired regularity bounds and decay estimates are the subject of the {companion paper} \cite{HRS2024confII}. 
\end{abstract}

\maketitle

\section{Introduction}
First-order mean-field systems have attracted enormous interest in recent years. At the microscopic level, the dynamics are described by the system of $N$ SDEs
\begin{align}\label{eq:SDE}
\begin{cases}
\displaystyle dx_i = \frac1N\sum_{1\leq j\leq N: j\neq i} \M\nabla\g(x_i-x_j)dt + \sqrt{\frac2\be}dW_i, \\
\displaystyle x_i|_{t=0} = x_i^\circ.
\end{cases}
\end{align}
Here, $\XN^\circ = (x_1^\circ,\ldots,x_N^\circ) \in (\R^\d)^N$ are the initial positions of the particles. The function $\g: \R^\d \rightarrow \R\cup \{\infty\}$ is the interaction potential, assumed to belong to the \emph{log/Riesz} class of singular, integrable potentials
\begin{align}
\g(x) \coloneqq \begin{cases} -\log|x|, & {\sf=0} \\ \frac1\sf|x|^{-\sf}, & {0<\sf<\d}. \end{cases}
\end{align}
The $\d\times\d$ matrix $\M$ has constant real entries. The two cases we will consider are $\M$ is antisymmetric (Hamiltonian/conservative) or $\M=-\I$ (gradient/dissipative), {though linear combinations of the two are permitted}. For convenience, we introduce the vector field notation $\k\coloneqq \M\nabla\g$. The parameter $\be\in (0,\infty]$ has the interpretation of inverse temperature, and $W_1,\ldots,W_N$ are independent Wiener processes in $\R^\d$. We refer to the introductions of \cite{JW2018, Serfaty2020, NRS2021, RS2021} and the surveys  \cite{JW2017_survey, CD2021, Golse2022ln} for a discussion of the motivation for and applications of these models. 
 
{Assuming that $\mu_N^\circ \coloneqq \frac1N\sum_{i=1}^N \delta_{x_i^\circ}$ suitably converges to $\mu^\circ$ as $N\rightarrow\infty$,} the empirical measure $\mu_N^t \coloneqq \frac1N\sum_{i=1}^N \delta_{x_i^t}$ is expected to converge to a solution of the \emph{mean-field equation}, which is the nonlinear PDE
\begin{equation}\label{eq:MFlim}
\begin{cases}
\p_t\mu + \div(\mu{\k}\ast\mu) = \frac1\beta\D\mu \\
\mu|_{t=0} = \mu^\circ,
\end{cases}
\qquad (t,x) \in \R_+ \times\R^\d.
\end{equation}
The \emph{mean-field limit} refers to this convergence. Related is the notion of \emph{propagation of chaos}, that  the $k$-point marginals of the law $f_N^t$ of the solution $\XN^t$ of \eqref{eq:SDE}, which satisfies the forward Kolmogorov equation
\begin{equation}\label{eq:FK}
\begin{cases}
\p_t f_N + \displaystyle\sum_{i=1}^N \div_{x_i}\Big(f_N \sum_{1\leq j\leq N : j\neq i} \k(x_i-x_j)\Big) = \frac1\beta\sum_{i=1}^N \Delta_{x_i}f_N \\
f_N|_{t=0} = f_N^\circ,
\end{cases}
\qquad (t,\XN) \in \R_+ \times (\R^\d)^N
\end{equation}
become $\mu^t$-chaotic as $N\rightarrow\infty$, where $\mu^t$ is a solution to \eqref{eq:MFlim}, assuming the initial law $f_N^\circ$ is $\mu^\circ$-chaotic. There is also the stronger notion of \emph{generation of chaos}, which asserts that $f_N^t$ becomes $\mu^t$-chaotic for large $N$ and $t$, even when $f_N^\circ$ is not $\mu^\circ$-chaotic. We refer to \cite{HM2014, RS2023lsi} for a more precise description of the relation between these notions.

The most powerful tools to prove mean-field limits/propagation of chaos when the potential $\g$ is singular, such as for the log/Riesz case, are the modulated energy  \cite{Serfaty2017, Duerinckx2016, Serfaty2020} (further developed in \cite{NRS2021}) and relative entropy of Jabin-Wang \cite{JW2016,JW2018} (previously, widely used for hydrodynamic limits and conservation laws, e.g. \cite{Yau1991, SR2009book} and \cite{DiPerna1979, Dafermos1979}), as well as their combination in the form of modulated free energy introduced by Bresch \emph{et al.} \cite{BJW2019crm, BJW2019edp, BJW2020}.\footnote{Other tools based on estimates for hierarchies have emerged recently {\cite{Lacker2023, LlF2023, BJS2022, hCR2023, BDJ2024}}. We will not discuss them in this work, but suffice it to say they are not as robust to the singularity of the interaction and temperature as modulated energy/entropy techniques.} We review these quantities in \cref{ssec:introME} below.

The modulated energy method is versatile to both the whole space and confined domains, such as the torus, at zero temperature $\be=\infty$ \cite{Duerinckx2016, Serfaty2020, NRS2021}. For positive temperature $\be<\infty$, a pure modulated energy approach was used to prove mean-field convergence by two of the authors \cite{RS2021} provided $\sf<\d-2$ (i.e., the interaction is sub-Coulomb). The approach works in both the whole space and the torus. However, for more singular interactions at positive temperature, one needs to utilize {relative} entropy, which has the disadvantage of requiring a strong regularity assumption on the solution of the limiting equation \eqref{eq:MFlim} in the form of $L^\infty$ estimates for the Hessian of the log density $\nabla^{\otimes2}\log\mu^t$.\footnote{The work of Lacker \cite{Lacker2023} is an exception in this regard.} On compact domains, such as the torus, these bounds are fairly straightforward, as one just needs to show that $\mu^t$ is positively bounded from below and establish $W^{2,\infty}$ bounds for $\mu^t$. However, in the full space $\R^\d$, the matter is more delicate, as no probability density can be bounded from below {by a positive constant}.


In recent work \cite{FW2023}, Feng-Wang extended the relative entropy method of \cite{JW2018} to prove local-in-time entropic propagation of chaos for the viscous vortex model, which corresponds to $\d=2$, $\M$ antisymmetric, and $\g(x)=-\log|x|$ in our notation.\footnote{We would expect that the result in fact holds for the Hamiltonian log case in any dimension $\d\ge 2$, as well.}\footnote{With this extension, one can also extend the Gaussian CLT for the fluctuations from \cite{WZZ2021} to the whole space. In the forthcoming \cite{HRSclt}, we use the results of the present paper to prove a CLT for the fluctuations of log/Riesz flows in the whole space.} The bulk of the work in \cite{FW2023} concerns establishing local-in-time spatially weighted  bounds for the solution of the mean-field equation up to second-order derivatives. For this, they rely on somewhat sophisticated tools for general parabolic equations; namely, Li-Yau  \cite{LY1986acta} and Hamilton-type \cite{Hamilton1993} estimates. In particular, there is the awkwardness of dealing with the solution having Gaussian-like spatial decay, hence the absolute log density grows quadratically as $|x|\rightarrow\infty$. Once one has these bounds, the remainder of the argument of \cite{JW2018} goes through.
 
The purpose of the present paper is to propose a simpler approach based on a transformation of the dynamics to self-similar coordinates, which converts all equations to ones with a quadratic confining potential. See the next subsection for further description. Such a transformation has been widely used in the PDE literature to study the asymptotic properties of dissipative equations (e.g., {\cite{Kapila1980, FMT1983, GK1987, Kavian1987, EK1988, MZ1997, GW2002, GW2005, BDP2006,BCL2009,BKM2010,CV2011asy,CD2014,SV2014}}). The relative entropy is invariant under this transformation, while the interaction potential $\g$ is renormalized by a time-dependent factor. In these coordinates, one can perform the relative entropy or modulated free energy method. Showing a closed estimate for the relative entropy or modulated free energy then reduces to $L^\infty$ bounds for the Hessian of the logarithm of the ratio of the mean-field density to a time-dependent thermal equilibrium measure. 

Such Hessian bounds (as well as higher-order derivatives) can be easily be proven, locally in time,  through a direct fixed point argument, energy estimates, or maximal principle, avoiding the appeals to more sophisticated tools, as in \cite{FW2023}.  With more work, building on ideas of the {authors and Chodron de Courcel} \cite{RS2021, CdCRS2023, CdCRS2023a}, one can show exponential-in-time rates of decay for this logarithmic ratio. This allows to show uniform-in-time propagation of chaos and even polynomial-in-time generation of chaos. The treatment of these regularity bounds and general analysis of the self-similar equation, for the first time covering both Hamiltonian and gradient drifts in the full range of  log and Riesz potentials, is the subject of the {companion paper} \cite{HRS2024confII}. 

\subsection{Modulated energy, entropy, and free energy}\label{ssec:introME}
Before proceeding further, we review the notions of modulated (free) energy.

Given $\mu\in\P(\R^\d)$, we define the modulated energy of the configuration $\XN \in (\R^\d)^N$ by 
\begin{equation}\label{eq:ME}
F_N(\XN, \mu) \coloneqq \hal \int_{(\R^{\d})^2\setminus\triangle} \g(x-y) d\Big(\frac1N\sum_{i=1}^N \delta_{x_i}- \mu\Big)^{\otimes 2} (x,y),
\end{equation}
where $\triangle$ denotes the diagonal of $(\R^\d)^2$.  This object first appeared as a next-order electric energy in \cite{SS2015,RS2016,PS2017} and was subsequently used in the dynamics context in \cite{Duerinckx2016,Serfaty2020} and following works---in the spirit of Brenier's modulated energy \cite{Brenier2000}.  The modulated energy {$F_N$} is the total interaction of the system of $N$ discrete charges located at $\XN$ against a negative (neutralizing) background charge $\mu$, with the infinite self-interaction of the points removed. As shown in the aforementioned prior works, $F_N$ is not necessarily positive (see \cref{lem:MEpos} below); however, it effectively acts as a squared distance between the empirical measure $\frac1N\sum_{i=1}^N \delta_{x_i}$ and $\mu$.

The normalized relative entropy between two probability densities $f_N, g_N$ on $(\R^\d)^N$ is given by\footnote{Throughout this paper, we abuse notation by using the same symbol to denote both a measure and its density.}
\begin{align}\label{eq:RE}
H_N(f_N\vert g_N)\coloneqq \frac{1}{N} \int_{(\R^\d)^N} \log\left(\frac{f_N}{g_N}\right) df_N.
\end{align}
Given a reference probability density $\mu$ on $\R^\d$, we may now define the modulated free energy, as introduced in \cite{BJW2019crm,BJW2019edp,BJW2020} by 
\begin{align}\label{eq:MFE}
E_N(f_N, \mu) \coloneqq \frac1\beta H_N(f_N\vert \mu^{\otimes N}) +\mathbb{E}_{f_N}\left[F_N(\XN,\mu)\right],
\end{align}
where $\mathbb{E}_{f_N}$ denotes the expectation with respect to the measure $f_N = \mathrm{Law}(X_N)$.

The average modulated energy satisfies the differential inequality (see \cite{RS2021})
\begin{multline}\label{eq:introdtME}
\frac{d}{dt}\E_{f_N^t}[\Fr_N(\XN,\mu^t)] \le \frac{1}{\be}\E_{f_N^t}\Bigg[\int_{(\R^\d)^2\setminus\triangle} \D\g(x-y)d\Big(\frac1N\sum_{i=1}^N \delta_{x_i}-\mu^t\Big)^{\otimes 2}(x,y)\Bigg]\\
+\E_{f_N^t}\Bigg[\frac1N\sum_{i=1}^N \PV \ \nabla\g\ast\Big(\frac1N\sum_{j=1}^N \delta_{x_j}-\mu^t\Big)(x_i) \cdot \PV \ \k\ast\Big(\frac1N\sum_{j=1}^N \delta_{x_j}-\mu^t\Big)(x_i)\Bigg] \\
- \frac12\E_{f_N^t}\Bigg[\int_{(\R^\d)^2\setminus\triangle}(u^t(x)-u^t(y))\cdot \nabla\g(x-y)d\Big(\frac1N\sum_{i=1}^N \delta_{x_i}-\mu^t\Big)^{\otimes 2}(x,y)\Bigg],
\end{multline}
where $\PV$ denotes principal value, $u^t \coloneqq -\k\ast\mu^t$. If $\sf<\d-2$, then the first term on the right-hand side is nonpositive up to $o_N(1)$ error, while the second term is nonpositive (vanishes if $\M$ is antisymmetric). The third term, called the commutator term, may be estimated in terms of the modulated energy using a functional inequality \cite{LS2018, Serfaty2020, Rosenzweig2020spv, Serfaty2023, Rosenzweig2021ne, NRS2021, RS2021} which bounds it by $ C \|\nabla u^t\|_{L^\infty} F_N(\XN, \mu^t)$. This allows to close the Gr\"onwall loop for $\E_{f_N^t}[\Fr_N(\XN,\mu^t)]$, as was done in \cite{RS2021}. Importantly, one only needs to control $\|\nabla u^t\|_{L^\infty}=  \|\nabla\k\ast\mu^t\|_{L^\infty}$,\footnote{In the Coulomb case, this can be weakened to controlling just $\|\mu^t\|_{L^\infty}$ \cite{Rosenzweig2022,Rosenzweig2022a}.} which is well-established in the PDE literature, e.g. \cite{Wolibner1933, Yudovich1963, CMT1994, CW1999qg, LZ2000, BKM2010, CCCGW2012, BLL2012, SV2014, BIK2015, CJ2021}. If $\sf\ge \d-2$, then the first term is no longer essentially nonpositive, and it is not clear how to proceed.

On the other hand, the relative entropy (see \cite{JW2018}) satisfies the differential inequality
\begin{multline}\label{eq:introdtRE}
\frac{d}{dt}H_N(f_N^t \vert (\mu^t)^{\otimes N}) \le  -\frac{1}{\beta}I_N({f}_N^{t} \vert (\mu^{t})^{\otimes N})\\
+\frac12\E_{f_N^t}\Big[\int_{(\R^\d)^2\setminus\triangle}\Big(\nabla\log(\mu^t)(x)-\nabla\log(\mu^t)(y)\Big)\cdot\k(x-y)\Big(\frac1N\sum_{j=1}^N\delta_{{x}_j} - \mu^{t}\Big)^{\otimes 2}(x,y)\Big]\\
+\E_{f_N^t}\Big[\int_{(\R^\d)^2\setminus\triangle} \div\k(x-y)\Big(\frac1N\sum_{j=1}^N\delta_{{x}_j} - \mu^{t}\Big)^{\otimes 2}(x,y) \Big],
\end{multline}
where $I_N$ denotes the $1/N$ normalized relative Fisher information, which is nonnegative. The simplest setting is when $\div\k=0$, which holds if $\M$ is antisymmetric, removing the third line. The difficulty in establishing a Gr\"{o}nwall relation is to control the second line, which again is a commutator term, in terms of the relative entropy. This can be done, as was shown in \cite{JW2018} (see \cref{lem:REcomm} below), but only when $\sf=0$. More importantly for our purposes, this requires control of $\|\nabla^{\otimes 2}\log\mu^t\|_{L^\infty}$ which is not obvious to establish in the whole space, given the decay of solutions as $|x|\rightarrow\infty$. For instance, if $\mu^t(x) \approx e^{-c|x|^{2+\epsilon}}$, then such a bound would be impossible.

As observed first by  Bresch \emph{et al.} \cite{BJW2019crm, BJW2019edp,BJW2020}, combining the relations \eqref{eq:introdtME}, \eqref{eq:introdtRE}, one sees that in the gradient case $\M=-\I$, the first term on the right-hand side of \eqref{eq:introdtME} cancels with $\frac1\be$ times the last term on the right-hand side of \eqref{eq:introdtRE}.  The second term term in \eqref{eq:introdtME} may be recombined with the relative Fisher information term in \eqref{eq:introdtRE} into a total nonpositive expression, leading to
\begin{multline}\label{eq:introdtMFE}
\frac{d}{dt}E_N(f_N^t, \mu^t) \le  -\frac{1}{\be^2}{I_N(f_N^t \vert \Q_{N,\be}(\mu^t))} \\
-\frac12\E_{f_N^t}\Big[\int_{(\R^\d)^2\setminus\triangle}\Big(u^t(x)-u^t(y)\Big)\cdot\nabla\g(x-y)\Big(\frac1N\sum_{j=1}^N\delta_{{x}_j} - \mu^{t}\Big)^{\otimes 2}(x,y) \Big],
\end{multline}
where now $u^t \coloneqq \frac1\be\nabla \log\mu^t + \nabla\g\ast\mu^t$ {and
\begin{align}
d\Q_{N,\be}(\mu^t)(\XN) \coloneqq \frac{1}{\K_{N,\be}(\mu^t)}e^{-\be N\Fr_N(\XN, \mu^t)}d(\mu^t)^{\otimes N}(\XN)
\end{align}
is the $\mu^t$-modulated Gibbs measure, with $\K_{N,\be}(\mu^t)$ the normalization factor or partition function.} To establish a Gr\"{o}nwall relation, one again needs to just estimate the commutator term in terms of the relative entropy, modulated energy, or a linear combination of the two,\footnote{In this regard, we are strongly using the repulsive nature of the interaction.} which can be done with the aforementioned functional ``commutator" inequality. This comes at the cost of needing a bound for $\|\nabla u^t\|_{L^\infty}$, which, by triangle inequality, again requires a bound for $\|\nabla^{\otimes 2}\log\mu^t\|_{L^\infty}$; the second term  {in $\nabla u^t$} can be straightforwardly estimated in terms of $\|\mu^\circ\|_{W^{2,\infty}}$, so there is no cancellation between the entropic and potential terms to be exploited.

The punchline of the preceding discussion is that entropic terms require $\dot{W}^{2,\infty}$ control on $\log\mu^t$, which is not guaranteed in the whole space for an arbitrary smooth, rapidly decaying function. However, this difficulty is, in some sense, an artifact of being in the wrong frame of reference: one has to take advantage of the asymptotic form of solutions to equation \eqref{eq:MFlim} in the large time limit. This leads us to the self-similar transformation.

\subsection{The self-similar coordinates}\label{ssec:intross}
Let $\mu^t$ be a solution of the mean-field equation \eqref{eq:MFlim}. For $x\in\R^\d$ and $t\ge 0$, self-similar coordinates are defined by
\begin{align}
\xi \coloneqq \frac{x}{\sqrt{t+1}}, \ \tau \coloneqq \log(t+1) \qquad &\Longleftrightarrow  \qquad x = e^{\tau/2}\xi, \ t = e^{\tau} -1 \label{eq:mfssvar} \\
\mu(t,x) \eqqcolon (t+1)^{-\d/2} \bmu(\log(t+1), x/\sqrt{t+1}) \qquad &\Longleftrightarrow \qquad e^{\d\tau/2}\mu(e^{\tau}-1, e^{\tau/2}\xi) = \bar{\mu}(\tau,\xi). \label{eq:bmudef}
\end{align}
Note that this transformation is mass-preserving. From the chain rule (see \cref{lem:sseqns} below), one checks that $\bmu^\tau$ solves the \emph{self-similar mean-field equation}
\begin{equation}\label{eq:MFlimss}
\begin{cases}
\p_{\tau}\bmu + \div\left(\bmu (\k_\tau\ast\bmu {-} \nabla(\frac14 |\xi|^2))\right) = \frac1\beta
\Delta\bmu\\
\bar\mu|_{\tau=0} = \mu^\circ
\end{cases}
\qquad (\tau, \xi) \in \R_+ \times \R^\d,
\end{equation}
where $\g_\tau \coloneqq e^{-\frac{\sf\tau}{2}}\g$ is the \emph{renormalized interaction potential} and $\k_\tau \coloneqq \M\nabla\g_\tau$.\footnote{We can also view it as an equation with the diffusion operator replaced by $\L_\be ^*\bmu \coloneqq \frac1\beta\Delta\bmu + \frac12\div(\xi\bmu)$, where $\L_\be^*$ is the \emph{Fokker-Planck operator}. It is customary to use the notation $\L_\be$ for the generator of the stochastic process of which $\bmu$ is the law, i.e. the Ornstein-Uhlenbeck operator. The Fokker-Planck operator is the adjoint of this operator with respect to the standard $L^2$ inner product, hence the notation $\L_\be^*$.} This way, we see the confinement potential $\frac14|\xi|^2$ appear. Remark that  $\g_0 = \g$ and if $\sf=0$, then $\g_\tau = \g$ for all $\tau\ge 0$.  
Similarly, if $f_N^t$ is a solution of the forward Kolmogorov equation \eqref{eq:FK}, then the {density}
\begin{align}\label{eq:bfNdef}
\bar{f}_N(\tau,\Xi_N) = e^{\d N\tau/2}f_N(e^{\tau}-1, e^{\tau/2}\Xi_N)
\end{align}
satisfies the \emph{self-similar forward Kolmogorov equation}
\begin{align}\label{eq:FKss}
\begin{cases}
\p_\tau \bar{f}_N + \displaystyle\sum_{i=1}^N \div_{\xi_i}\Big(\bar{f}_N \Big(\sum_{j\neq i} \k_\tau(\xi_i-\xi_j) {- \nabla_{\xi_i}(\frac14|\Xi_N|^2)}\Big)\Big) = \frac1\beta\D_{\Xi_N} \bar{f}_N \\
\bar{f}_N |_{t=0} = f_N^\circ,
\end{cases}
\qquad (\tau,\Xi_N) \in \R_+ \times (\R^\d)^N.
\end{align}
Thus, our original problem has been converted to one with a confining potential, at the cost of renormalizing the interaction potential by a time-dependent factor.

\begin{remark}\label{rem:rad}
If $\M$ is antisymmetric and $\bmu$ is radial, then the nonlinear term $\div(\bmu\k_{\tau}\ast\bmu) =0$ and the equation reduces to the linear Fokker-Planck equation. Equivalently, under the same assumptions equation \eqref{eq:MFlim} becomes the linear heat equation.
\end{remark}

The relative entropy is invariant under changing to self-similar coordinates. In contrast, the modulated energy transforms to 
\begin{align}\label{eq:bFNtaudef}
\bar{F}_N^\tau(\Xi_N,\bmu) + \frac{\tau}{4N}\indic_{\sf=0} &\coloneqq \int_{(\R^\d)^2\setminus\triangle}\g_\tau(\xi-\eta)d\Big(\frac1N\sum_{i=1}^N \delta_{\xi_i}-\bmu\Big)^{\otimes 2}(\xi,\eta) + \frac{\tau}{4N}\indic_{\sf=0} \nn\\
&= e^{-\sf\tau/2}\Fr_N(\Xi_N,\bmu)  + \frac{\tau}{4N}\indic_{\sf=0},
\end{align}
and therefore the modulated free energy transforms to
\begin{align}
\bar{E}_N^\tau(\bar{f}_N^\tau, \bmu^\tau) + \frac{\tau}{4N}\indic_{\sf=0} \coloneqq \frac1\be H_N(\bar{f}_N \vert \bmu^{\otimes N}) +  \E_{\bar{f}_N}\Big[\bar{\Fr}_N^\tau(\Xi_N,\bmu)\Big] + \frac{\tau}{4N}\indic_{\sf=0}.
\end{align}
We refer to \Cref{lem:REss,lem:MFEss} for the detailed computation. Just as with \eqref{eq:introdtME}, \eqref{eq:introdtRE}, \eqref{eq:introdtMFE}, each of these transformed quantities satisfies a dissipation identity amenable to a Gr\"onwall argument.

Focusing, in the interests of brevity, on just the modulated free energy when $\M=-\I$, one has (see \cref{lem:MFEdiss})
\begin{multline}\label{eq:introdtaubEN}
\frac{d}{d\tau}\bar{E}_N^\tau(\bar{f}_N^\tau, \bmu^\tau) \le -\frac{1}{\be^2}I_N(\bar{f}_N^\tau \vert \Q_{N,\be}^\tau(\bmu^\tau))- \frac{\sf}{2}\E_{\bar{f}_N^\tau}\Big[\bFr^\tau(\Xi_N,\bmu^\tau)\Big]  \\
-\frac12\E_{\bar{f}_N^\tau}\Big[\int_{(\R^\d)^2\setminus\triangle}(\bar{u}^\tau(\xi)-\bar{u}^\tau(\eta))\cdot\nabla\g_\tau(\xi-\eta)d\Big(\frac1N\sum_{i=1}^N \delta_{\xi_i}-\bmu^\tau\Big)^{\otimes 2}(\xi,\eta)\Big],
\end{multline}
where $\mathbb{Q}_{N,\be}^{\tau}(\bmu^\tau)$ is the modulated Gibbs measure given by
\begin{align}\label{eq:mGM}
d\mathbb{Q}_{N,\be}^{\tau}(\bmu^\tau)(\Xi_N) \coloneqq \frac{1}{\mathsf{K}_{N,\be}^\tau(\bmu^\tau)}e^{-\be N\bar{F}_N^\tau(\Xi_N,\bmu^\tau)}d(\bmu^\tau)^{\otimes N}(\Xi_N),
\end{align}
with partition function $\mathsf{K}_{N,\be}^\tau(\bmu^\tau)$, and the vector field $\bar{u}^\tau \coloneqq \frac1\be\nabla\log\bmu^\tau + \frac12\xi + \nabla\g\ast\bmu^\tau$ has gained the contribution of the confining potential $\frac14|\xi|^2$. To estimate the commutator term in \eqref{eq:introdtaubEN}, one again wants to use the aforementioned functional inequality, which requires bounding $\|\nabla \bar{u}^\tau\|_{L^\infty}$.
To do so,  one should not try to estimate each of the three terms in the expression for $\nabla\bar{u}^\tau$ separately. Rather, one should exploit a cancellation in their combination described as follows.

Since the relative entropy is strictly convex and $\hat{\g}\ge 0$, the time-dependent macroscopic free energy
\begin{align}\label{MFenergy}
\bar{\Ec}_\beta^\tau(\bmu) \coloneqq \hal \int_{(\R^\d)^2} \g_\tau(\xi-\eta) d\bmu(\xi) d\bmu(\eta)+ \int_{\R^\d}\frac14|\xi|^2 d\bmu(\xi)  + \frac{1}{\beta}\int_{\R^\d}(\log \bmu) d\bmu
\end{align}
is strictly convex and admits a unique minimizer $\bmu_\be^\tau$ in $\P(\R^\d)$, characterized by the relation
\begin{align}\label{eq:bmubetaurel}
\frac1\be\log\bmu_\be^\tau + \frac14|\xi|^2+ \g_\tau\ast\bmu_\be^\tau = Z_{\be}^\tau,
\end{align}
where $Z_{\be}^\tau \in\R$. Remark that if  $\sf=0$, then $\bar{\Ec}_\be^\tau$ is independent of time, and $\bmu_\be^\tau$ is the usual thermal equilibrium measure, which is the unique stationary state of equation \eqref{eq:MFlimss} among probability measures with finite free energy. We refer to \cite[Chap. 2]{SerfatyLN} for a description of the thermal equilibrium measure.  If on the other hand  $\sf>0$, then $\g_\tau \to 0$  as $\tau\rightarrow\infty$, hence  $\bmu_\be^\tau$ converges to the Gaussian
\begin{align}
d\bmu_\beta^\infty(\xi) \coloneqq (4\pi/\beta)^{-\frac{\d}{2}}e^{-\beta |\xi|^2/4}d\xi,
\end{align}
which minimizes the sum of the last two terms in \eqref{MFenergy}.

Taking the gradient of the left-hand side of the relation \eqref{eq:bmubetaurel}, which is zero, and subtracting from $\bar{u}^\tau$,  we observe the cancellation
\begin{align}
\bar{u}^\tau = \frac1\be\nabla\log\frac{\bmu^\tau}{\bmu_\be^\tau} + \nabla\g\ast(\bmu^\tau-\bmu_\be^\tau).
\end{align}
Each of the two terms may now be estimated separately. The second term is straightforward, while the first term is more subtle, as discussed below. 







\subsection{Main result}\label{ssec:introMR}
We come to the main result. In a nutshell, we show that if the self-similar mean-field equation \eqref{eq:MFlimss} admits a solution $\bmu^\tau$, such that the Hessian $\nabla^{\otimes 2}\log\frac{\bmu^\tau}{\bmu_\be^\tau}$ is in $L^\infty$ uniformly in $\tau \in [0,\log(T+1)]$, then the following holds:
\begin{enumerate}[(1)]
\item if $\sf=0$ and $\M$ is antisymmetric, the original relative entropy $H_N(f_N^t\vert (\mu^t)^{\otimes N})$ is bounded by $C_T(H_N(f_N^0 \vert (\mu^0)^{\otimes N}) + o_N(1))$ on {$[0,T]$};
\smallskip
\item if $\sf<\d-2$ and $\M$ is antisymmetric or $\sf<\d$ and $\M=-\I$, the modulated free energy $E_N(f_N^t \vert (\mu^t)^{\otimes N})$ is bounded by $C_T(E_N(f_N^0,\mu^0) + o_N(1))$ on {$[0,T]$};\smallskip
\item if $\int_0^\infty \|\nabla^{\otimes 2}\log\frac{\bmu^\tau}{\bmu_\be^\tau}\|_{L^\infty}d\tau<\infty$, then the preceding estimates hold uniformly in time;\smallskip
\item  if for $\sf<\d-2$ and $\M$ antisymmetric,  $\|\log\frac{\bmu^\tau}{\bmu_\be^\infty}\|_{L^\infty}$ is uniformly bounded in time, then the prefactor $C_T$ in point (1) vanishes polynomially fast as $T\rightarrow\infty$;\smallskip
\item if for $ \sf<\d$ and $\M = -\I$, the modulated Gibbs measure $\mathbb{Q}_{N,\be}^\tau(\bmu^\tau)$ from \eqref{eq:mGM} satisfies a logarithmic Sobolev inequality (LSI) uniformly in $N$ and $\tau$ (a uniform modulated LSI in the language of \cite{RS2023lsi}), then the dependence on the initial modulated free energy in point (2) vanishes polynomially fast as $T\rightarrow\infty$.
\end{enumerate}
As is well-known, by subadditivity of the relative entropy and the almost positivity of the modulated energy, the first two assertions imply propagation of chaos on the interval $[0,T]$. The third assertion implies uniform-in-time propagation of chaos. While the fourth and fifth assertions imply generation of chaos. The important takeaway is that the problem of propagation and generation of chaos in the whole space is reduced to a pure regularity question for a nonlinear Fokker-Planck equation, which can be established with a little effort. See \cref{ssec:introreg} below for further discussion.

\begin{mthm}\label{thm:main}
Suppose that equation \eqref{eq:MFlimss} admits a solution $\bmu\in C_w([0,T], \P(\R^\d))$, the space of functions on $[0,T]$ taking values in $\P(\R^\d)$ and continuous with respect to weak convergence, such that $\|\nabla^{\otimes 2}\log\frac{\bmu}{\bmu_\be}\|_{L^\infty([0,T], L^\infty)} < \infty$. Let $f_N^t$ be an entropy solution to \eqref{eq:FK}.\footnote{In fact, the SDE system \eqref{eq:SDE} has a unique strong solution if $\sf\le \d-2$ \cite[Section 4]{RS2021}. In the gradient case, this result is extendable to the full range $\sf<\d$ \cite{hC2022}.} Define the quantity
\begin{align}
\Ec_N(f_N^t,\mu^t) \coloneqq E_N(f_N^t,\mu^t) + \os_N^t(1) \ge 0,
\end{align}
where $\os_N^t(1)$ is a correction to obtain a nonnegative quantity and depends on $\mu^t$ only through $\|\mu^t\|_{L^\infty}$ and vanishes as $N\rightarrow\infty$. See \eqref{eq:osNdef} for the precise definition. Also, introduce the seminorm
\begin{align}
\|v\|_{*} \coloneqq \begin{cases} \|\nabla v\|_{L^\infty}, & {\sf\geq \d-2} \\ \|\nabla v\|_{L^\infty} + \|(-\Delta)^{\frac{\d-\sf}{4}}v\|_{L^{\frac{2\d}{\d-2-\sf}}}, & {\sf<\d-2}. \end{cases}
\end{align} 
Then the following hold.

\textbullet \ Suppose that $\M$ is antisymmetric. If $\sf=0$, then for $t\in [0,T]$,
\begin{multline}\label{eq:mainRE}
H_N({f}_N^{t} \vert (\mu^{t})^{\otimes N}) \le {\exp\Big({\int_0^{\log(t+1)}\Big(\|\nabla^{\otimes 2}\log\frac{\bmu^{\tau'}}{\bmu_\be^\infty}\|_{L^\infty} - e^{-2\|\log\frac{\bmu^{\tau'}}{\bmu_\be^\infty}\|_{L^\infty}}\Big)d\tau'  }\Big)} H_N(f_N^0 \vert (\mu^0)^{\otimes N})  \\
+ \frac{C}{N}\int_0^{\log(t+1)} {\exp\Big({\int_{\tau'}^{\log(t+1)}\Big(\|\nabla^{\otimes 2}\log\frac{\bmu^{\tau''}}{\bmu_\be^\infty}\|_{L^\infty} - e^{-2\|\log\frac{\bmu^{\tau''}}{\bmu_\be^\infty}\|_{L^\infty}}\Big)d\tau''  }\Big)}\| \nabla^{\otimes 2}\log\frac{\bmu^{\tau'}}{\bmu_\be^\infty}\|_{L^\infty}   d\tau',
\end{multline}
where $C>0$ is a constant depending only on $\d$. Define
\begin{align}
\ka \coloneqq \inf_{[0,\log(T+1)]} \min\Big(e^{-2\|\log\frac{\bmu^{\tau'}}{\bmu_\be^\infty}\|_{L^\infty}}, \frac{\sf}{4}\Big) \in [0,1).
\end{align}
If $0<\sf<\d-2$, then there exists $\al=\al(\d,\sf)>0$ such that for $t\in [0,T]$,
\begin{multline}\label{eq:mainMFE1}
{\Ec}_N({f}_N^t, \mu^t)  \le  \frac{e^{C\int_0^{\log(t+1)}  \|\bar{u}^{\tau'}\|_{*}d\tau'}}{(t+1)^{\ka}}\Bigg[{\Ec}_N({f}_N^0, \mu^0) + C^t[1-(t+1)^{-\frac{\sf}{4}}]N^{-\al}\\
+\frac{C^t[1-(t+1)^{-\frac{\sf}{4}}]}{\be\sf} N^{-\al}\Bigg],
\end{multline}
where $C>0$ depend only on $\d,\sf$ and $C^t$ additionally depends on $t$ only through $\|\bmu^\tau\|_{L^\infty}$ on the interval $[0,\log(t+1)]$.

\textbullet \ Suppose that $\M=-\I$ and $0\le \sf<\d$. Let $C_{LS,N}^\tau \in [0,\infty]$ denote the LSI constant of the modulated Gibbs measure ${\Q}_{N,\be}^\tau(\bmu^\tau)$, and set $\ka\coloneqq \inf_{\tau\ge 0} \frac{1}{C_{LS,N}^\tau\be}$. Then
\begin{multline}\label{eq:mainMFE2}
\Ec_N(f_N^t,\mu^t) \le {e^{C\int_0^{\log(t+1)} \|\bar{u}^{\tau'}\|_{*}d\tau'}}\Bigg[\frac{\Ec_N(f_N^0,\mu^0)}{(t+1)^{\ka}} + {C^t[1-(t+1)^{-\ka}]}{N}^{-\al}\\
+ \frac{C^t\ka\be[(t+1)^{\ka-\frac{\sf}{2}}-1]}{(\ka-\frac{\sf}{2})(t+1)^{{\ka}}}N^{-\al} +\frac{C^t\sf[(t+1)^{\ka-\frac{\sf}{2}}-1]}{(\ka-\frac{\sf}{2})(t+1)^{{\ka}}}N^{-\al}\Bigg],
\end{multline} 
where $\al,C,C^t$ are as above. 
\end{mthm}

A sharper version of the estimates in \cref{thm:main}, in particular the explicit form of the exponent $\al$, is contained in the statements of \Cref{prop:REgron,prop:MFEgron1,prop:MFEgron2}. We have opted to present here a cruder form so as to make the result easier to digest.


We have essentially sketched the proof of \cref{thm:main} in the previous subsection. After passing to self-similar coordinates, one just needs to establish a Gr\"{o}nwall relation from the dissipation inequality following \cite{JW2018} (relative entropy) and  \cite{BJW2019edp, Serfaty2020} (modulated free energy). Once such a closed estimate for the self-similar relative entropy or modulated free energy has been obtained, one simply reverts back to the original coordinate system, using the invariance properties of these quantities. By carefully tracking the dependence of the estimates on the mean-field density, one obtains the uniformity in time. To obtain generation of chaos, one exploits the relative Fisher information in the form of an LSI for $(\bmu^\tau)^{\otimes N}$ (Hamiltonian) or $\mathbb{Q}_{N,\be}^\tau(\bmu^\tau)$ (gradient), as was done in the Hamiltonian $\sf=0$ case in \cite{GlBM2021} and the general gradient Riesz case in \cite{RS2023lsi}. This yields an exponential-in-$\tau$ decay on the dependence of the initial relative entropy or modulated free energy. Returning to the original coordinate system, this decay becomes polynomial in $t$. We mention that while an LSI for $(\bmu^\tau)^{\otimes N}$ is easy to establish by tensorization and perturbing around the Gaussian LSI \cite{Gross1975}, an LSI for $\mathbb{Q}_{N,\be}^\tau(\bmu^\tau)$ is quite difficult to establish, even locally in $N$ and $\tau$, and is only known in the $\d=1$ log/Riesz case \cite{CL2020,RS2023lsi}, thanks to the convexity of the interaction. 

We close this subsection with the following additional remarks. 

\begin{remark}
As pointed out to the first author by Daniel Lacker, one could relax the finite initial relative entropy assumption in \cref{thm:main}, since the regularizing properties of the flow ensure this for positive times. 
\end{remark}

\begin{remark}\label{rem:MFEham}
Implicit in \cref{thm:main} is a new  extension of the modulated free energy to the Hamiltonian case, which works for both the whole space or torus. This allows to obtain (polynomial-in-time) generation of chaos in the Hamiltonian case $\sf\in (0,\d-2)$ through modulated free energy, conditional on a uniform-in-$\tau$ bound for $\|\log\frac{\bmu^\tau}{{\bmu_\be^\infty}}\|_{L^\infty}$ and an exponential-in-$\tau$ rate of decay for $\|\nabla^{\otimes 2}\log\frac{\bmu^\tau}{\bmu_\be^\tau}\|_{L^\infty}$. Previously, generation of chaos was only known for the $\sf=0$ case on the torus using a pure relative entropy approach \cite{GlBM2021}. The needed regularity/decay bounds are the subject of \cite{HRS2024confII}, but we mention that these bounds were already established on the torus by the authors and Chodron de Courcel in \cite{CdCRS2023}. Consequently, by following the argument of this paper, one has a complete proof of (exponential-in-time) generation chaos for the Hamiltonian case $\sf<\d-2$ on the torus through modulated free energy. See \cref{rem:MFEhamLSItorus} below for further elaboration. {We remark that, as shown in \cite{RS2021}, modulated energy alone suffices for the Hamiltonian case $\sf\in (0,\d-2)$, uniformly in time. But this only yields propagation of chaos in a weak sense (i.e., convergence in negative-order Sobolev norms), as opposed to the stronger entropic sense. Moreover, it is not clear how to establish generation of chaos using only modulated energy.} 
\end{remark}

\begin{remark}\label{rem:atrct}
One could ask whether a version of \cref{thm:main} holds for the attractive log gas, which is the only attractive case in our family of interactions for which propagation of chaos is known, in particular, for which modulated free energy is implementable \cite{BJW2019crm, CdCRS2023a}. The self-similar transformation works just as well. But the real issue is the extension to the whole space of the modulated logarithmic HLS inequality, needed to establish the coercivity of the modulated free energy. This is an interesting problem we shall consider elsewhere.
\end{remark}

\subsection{Regularity estimates}\label{ssec:introreg}
As previously stated, \cref{thm:main} yields propagation of chaos for the system \eqref{eq:SDE} in the whole space, conditional on solutions of the self-similar mean-field equation \eqref{eq:MFlimss} satisfying certain regularity estimates. Uniform bounds on $\|\bmu^\tau\|_{L^\infty}$ are an elementary consequence of the ultracontractivity for the original $\mu^t$, converted to self-similar coordinates, shown in \cite{CL1995} (Hamiltonian) and \cite{BIK2015, RS2021} (gradient). In fact, by combining energy arguments and entropy methods similar to \cite{CdCRS2023}, one can establish exponential-in-time rates of decay for $\bmu^\tau-\bmu_\be^\tau$ in Sobolev spaces with polynomial spatial weights (cf. \cite{GW2002, GW2005}). But this still leaves the question of log density bounds.

The main observation is that $\ws^\tau \coloneqq \log\frac{\bmu^\tau}{\bmu_\be^\tau}$ satisfies the nonlinear PDE
\begin{multline}\label{eq:dtws}
\p_\tau\ws^\tau = \frac1\beta\Delta\ws^\tau - \frac1\be|\nabla \ws^\tau|^2 + 2\nabla\ws^\tau\cdot (\nabla\g_\tau-\k_\tau)\ast\bmu^\tau -\div\k_\tau\ast(\bmu^\tau-\bmu_\be^\tau) \\
- \nabla\log\bmu_\be^\tau\cdot\k_\tau\ast(\bmu^\tau-\bmu_\be^\tau)-\div(\nabla\g_\tau+\k_\tau)\ast\bmu_\be^\tau -\nabla\log\bmu_\be^\tau \cdot (\nabla\g_\tau+\k_\tau)\ast\bmu_\be^\tau,
\end{multline}
which is amenable to direct maximum principle arguments. Note that sums involving $(\nabla\g_\tau+\k_\tau)\ast\bmu_\be^\tau$ vanish if $\M=-\I$, while, by radial symmetry, the inner product of $\k_\tau\ast\bmu_\be^\tau$ with the gradient of a radial function vanishes. In fact, if one is just interested in local-in-time estimates, then these can be established fairly easily through the energy method and Sobolev embedding. Note that for short times, $\bmu_\be^\tau$ is well approximated by $\bmu_\be^0$, while for large times, $\bmu_\be^\tau$ is well approximated by $\bmu_\be^\infty$. We report on these findings in the forthcoming \cite{HRS2024confII}.

\begin{remark}
In the Hamiltonian case, it is easy to see that solutions satisfy the desired regularity estimates if the initial data is radially symmetric. Indeed, by \cref{rem:rad}, the equation reduces to the Fokker-Planck equation with quadratic confinement; equivalently, $\frac{\bmu^\tau}{{\bmu_\be^\infty}}$ solves the Ornstein-Uhlenbeck equation. From the Mehler formula \cite[Section 2.7.1]{BGL2014}, one can deduce that if $\frac{\bmu^\circ}{{\bmu_\be^\infty}}\ge c>0$, then $\frac{\bmu^\tau}{{\bmu_\be^\infty}}\ge c$ for all $\tau\ge 0$. Moreover, $\|\nabla^{\otimes k}\frac{\bmu^\tau}{{\bmu_\be^\infty}}\|_{L^\infty} = O(e^{-k\tau})$ as $\tau\rightarrow\infty$. By the chain rule, this implies an exponential-in-time rate of decay for $\|\nabla^{\otimes 2}\log\frac{\bmu^\tau}{\bmu_\be^\infty}\|_{L^\infty}$.
\end{remark}

\begin{remark}[Added in proof]\label{rem:Monmarche}
During the finalization of this manuscript, Monmarch\'{e} \emph{et al.} uploaded to arXiv their preprint \cite{MRW2024}. They establish global bounds for the Hessian of the log of the Radon-Nikodym derivative of the mean-field solution relative to the Gaussian $\mu_\be^\infty$ for a sub-range of Hamiltonian and gradient Riesz flows with quadratic confinement in the whole space and exponential-in-time decay estimates for this quantity for the Hamiltonian case $\sf<\d-1$. Since $\g_\tau=\g$ for $\sf=0$, their work shows in this case that the required regularity estimates for our \cref{thm:main} are satisfied. However, for $\sf>0$, their results are not applicable to our setting, since they do not account for the renormalization of $\g$. A more detailed comparison between our approach to estimates for the self-similar mean-field equation and theirs for the usual mean-field equation with quadratic confinement is deferred to \cite{HRS2024confII}.
\end{remark}

\subsection{Organization of paper}\label{ssec:introorg}
Let us close out the introduction with some comments on the organization of the remaining body of the paper. In \cref{sec:SS}, we present the computations for the self-similar transformation. \cref{ssec:SSeqns} treats the mean-field and forward Kolmogrov equations, \cref{ssec:SSre} treats the relative entropy, and \cref{ssec:SSmfe} treats the modulated (free) energy. In \cref{sec:MAIN}, we give the proof of \cref{thm:main}. \cref{ssec:MAINre} treats the relative entropy portion of the theorem, with \cref{prop:REgron} the main result; and \cref{ssec:MAINmfe} treats the modulated free energy portion, with \cref{prop:MFEgron1,prop:MFEgron2} the main results.

\subsection{Acknowledgments}
The authors thank Jiaoyang Huang for discussion that inspired the results of this paper. The first author thanks Pierre Monmarch\'{e} and Songbo Wang for discussion about their recent work \cite{MRW2024}.

\section{Transformation to self-similar coordinates}\label{sec:SS}
In this section, we give the detailed computations for how equations \eqref{eq:MFlim} and \eqref{eq:FK}, the modulated energy \eqref{eq:ME}, relative entropy \eqref{eq:RE}, and modulated free energy \eqref{eq:MFE} transform under the change to self-similar coordinates.

\subsection{Mean-field and forward Kolmogorov equations}\label{ssec:SSeqns}
We first show that $\mu^t, f_N^t$ respectively solve the mean-field \eqref{eq:MFlim} and forward Kolmogorov \eqref{eq:FK} equations if and only if their self-similar transformations $\bmu^\tau, \bar{f}_N^\tau$ respectively solve their self-similar forms \eqref{eq:MFlimss} and \eqref{eq:FKss}.

\begin{lemma}\label{lem:sseqns}
Let $T>0$. $\mu$ solves equation \eqref{eq:MFlim} on $[0,T]$ if and only if $\bmu$ solves equation \eqref{eq:MFlimss} on $[0,\log(T+1)]$. $f_N$ solves equation \eqref{eq:FK} on $[0,T]$ if and only if $\bar{f}_N$ solves equation \eqref{eq:FKss} on $[0,\log(T+1)]$. 
\end{lemma}
\begin{proof}
We only present the ``only if'' direction, leaving the ``if'' as an exercise for the reader. Recalling \eqref{eq:bmudef} and \eqref{eq:MFlim}, the chain rule implies
\begin{align}
\p_\tau\bmu(\tau,\xi) &= \frac{\d}{2}e^{\d\tau/2}\mu(e^\tau - 1, e^{\tau/2}\xi)  + e^{\d\tau/2 + \tau}(\p_t\mu)(e^\tau - 1, e^{\tau/2}\xi) + e^{\d\tau/2}(\nabla_x\mu)(e^\tau - 1, e^{\tau/2}\xi) \cdot \frac12 e^{\tau/2}\xi \nn\\
&= \frac{\d}{2}\bmu(\tau,\xi) + e^{\d\tau/2 + \tau}\Big(\frac1\be\Delta_x\mu -\div_x(\mu\k\ast\mu)\Big)(e^{\tau}-1,e^{\tau/2}\xi) + \frac12\xi\cdot\nabla_{\xi}\bmu(\tau,\xi) \nn\\
&=\frac{\d}{2}\bmu(\tau,\xi) + \frac1\be\D_{\xi}\bmu(\tau,\xi)-\div_{\xi}(\bmu\k_\tau\ast\bmu)(\tau,\xi) + \frac12\xi\cdot\nabla_{\xi}\bmu(\tau,\xi).
\end{align}
Above, we have implicitly used the identities $\nabla_{\xi}\Big(\mu(e^{\tau/2}\xi)\Big) = e^{\tau/2}(\nabla_x \mu)(e^{\tau/2}\xi)$ and $(\nabla_x\g\ast\mu)(e^{\tau/2}\xi) = e^{-(1+\sf-\d)\tau/2}\Big(\nabla_{\xi}\g\ast \mu(e^{\tau/2}\cdot)\Big)(\xi)$. Rewriting the sum of the first and last terms as $\div_{\xi}(\frac12\xi\bmu)(\tau,\xi)$ completes the proof of the first assertion.

Similarly, recalling \eqref{eq:bfNdef} and \eqref{eq:FK}, the chain rule implies
\begin{align}
\p_\tau\bar{f}_N(\tau,\Xi_N)  &= \frac{\d N}{2}e^{\d N\tau/2}f_N(e^{\tau}-1,e^{\tau/2}\Xi_N) + e^{\d N\tau/2+\tau}(\p_t f_N)(e^\tau-1,e^{\tau/2}\Xi_N) \nn\\
&\ph +  e^{\d N\tau/2}(\nabla_{\XN}f_N)(e^\tau-1,e^{\tau/2}\Xi_N)\cdot \frac12e^{\tau/2}\Xi_N  \nn\\
&=\frac{\d N}{2}\bar{f}_N(\tau,\Xi_N) + e^{\d N\tau/2+\tau}\sum_{i=1}^N \div_{x_i}\Big(f_N\frac1N\sum_{j\ne i}\k(x_i-x_j)\Big)(e^\tau -1, e^{\tau/2}\Xi_N) \nn\\
&\ph + \nab_{\Xi_N}\bar{f}_N(\tau,\Xi_N)\cdot\frac12\Xi_N \nn\\
&=\frac{\d N}{2}\bar{f}_N(\tau,\Xi_N) + \sum_{i=1}^N \div_{\xi_i}\Big(\bar{f}_N \sum_{j\ne i}\k_\tau(\xi_i-\xi_j)\Big)(\tau,\Xi_N) +  \nab_{\Xi_N}\bar{f}_N(\tau,\Xi_N)\cdot\frac12\Xi_N .
\end{align}
Above, we have implicitly used the identity $\k(e^{\tau/2}\xi_i - e^{\tau/2}\xi_j) = e^{-(1+\sf)\tau/2}\k(\xi_i-\xi_j)$. Combining the first and last terms of the final line by product rule completes the proof of the second assertion.
\end{proof}

\subsection{Relative entropy}\label{ssec:SSre}
As claimed in the introduction, the relative entropy is invariant under transformation to self-similar coordinates.

\begin{lemma}\label{lem:REss}
Let $f_N, {g_N}\in \P((\R^\d)^N)$. For $\tau\ge 0$, define $\bar{f}_N(\Xi_N) \coloneqq e^{\d N\tau/2}f_N(e^{\tau/2}\Xi_N)$ and {$\bar{g}_N(\Xi_N) \coloneqq e^{\d N\tau/2}g_N(e^{\tau/2}\Xi_N)$}. Then
\begin{align}\label{eq:REss}
H_N(f_N \vert g_N) = H_N(\bar{f}_N \vert \bar{g}_N).
\end{align}
\end{lemma}
\begin{proof}
For $\XN\in(\R^\d)^N$, writing $f_N(\XN) = e^{-\d N\tau/2}\bar{f}_N(e^{-\tau/2}\XN)$ and $g_N(\XN) = e^{-\d N\tau/2}\bar{g}_N(e^{-\tau/2}\XN)$, we compute
\begin{align}
H_N(f_N \vert g_N) &= \frac{1}{N}\int_{(\R^\d)^N} e^{-\d N\tau/2}\bar{f}_N(e^{-\tau/2}\XN)\log\left(\frac{\bar{f}_N(e^{-\tau/2}\XN)}{\bar{g}_N(e^{-\tau/2}\XN)}\right)d\XN\nn\\
&=\frac{1}{N}\int_{(\R^\d)^N}\bar{f}_N(\Xi_N)\log\left(\frac{\bar{f}_N(\Xi_N)}{\bar{g}_N(\Xi_N)}\right)d\Xi_N\nn\\
&= H_N(\bar{f}_N \vert \bar{g}_N), 
\end{align}
where the penultimate line follows from the change of variable $\Xi_N = e^{-\tau/2}\XN$.
\end{proof}

\subsection{Modulated (free) energy}\label{ssec:SSmfe}
We compute the transformation of the modulated free energy under the change to similarity variables. In contrast to the relative entropy, the modulated energy is only invariant if $\sf=0$.

\begin{lemma}\label{lem:MFEss}
Let $\mu \in \P(\R^\d)$ and $f_N \in \P((\R^\d)^N)$. For $\tau\ge 0$, define $\bmu(\xi) \coloneqq e^{\d\tau/2}\mu(e^{\tau/2}\xi)$ and $\bar{f}_N(\Xi_N) \coloneqq e^{\d N\tau/2}f_N(e^{\tau/2}\Xi_N)$. Then
\begin{align}\label{eq:MFEss1}
\E_{f_N}\Big[\Fr_N(\XN,\mu)\Big] = \E_{\bar{f}_N}\Big[\Fr_N^\tau(\Xi_N,\bmu)\Big] + \frac{\tau}{4N}\indic_{\sf=0}
\end{align}
and
\begin{align}\label{eq:MFEss2}
E_N(f_N, \mu) =  \frac1\be H_N(\bar{f}_N \vert \bmu^{\otimes N}) +  \E_{\bar{f}_N}\Big[\bar{\Fr}_N^\tau(\Xi_N,\bmu)\Big] + \frac{\tau}{4N}\indic_{\sf=0} \eqqcolon \bar{E}_N^\tau(\bar{f}_N, \bmu) + \frac{\tau}{4N}\indic_{\sf=0}.
\end{align}
Consequently, if
\begin{align}
\Ec_N(f_N,\mu) &\coloneqq E_N(f_N, \mu) + \mathsf{C}_1\frac{\log(N\|\mu\|_{L^\infty})}{N}\indic_{\sf=0} + \mathsf{C}_2\|\mu\|_{L^\infty}^{\ga} N^{\la},
\end{align}
for real constants $\mathsf{C}_1,\mathsf{C}_2,\ga,\la>0$, then
\begin{multline}\label{eq:MFEss3}
\Ec_N(f_N,\mu) = \frac1\be H_N(\bar{f}_N \vert \bmu^{\otimes N}) +\E_{\bar{f}_N}\Big[ \bFr^\tau(\Xi_N,\bmu)\Big] + \mathsf{C}_2e^{-\frac{\d\ga\tau}{2}}\|\bmu\|_{L^\infty}^{\ga}N^{\la}\\
+\Big(\mathsf{C}_1\frac{\log(N\|\bmu\|_{L^\infty})}{ N}  + \frac{\tau}{N}\Big(\frac14-\frac{\d\mathsf{C}_1}{2}\Big)\Big)\indic_{\sf=0}  \eqqcolon \bEc^\tau(\bar{f}_N,\bmu) + \frac{\tau}{N}\Big(\frac14-\frac{\d\mathsf{C}_1}{2}\Big)\indic_{\sf=0}.
\end{multline}
\end{lemma}

\begin{remark}
If $\mathsf{C}_1 = \frac{1}{2\d}$, $\ga=\frac{\sf}{\d}$, and $\la = {\frac{\sf}{\d}-1}$, which, by \cref{lem:MEpos} below, we are allowed to take if $\d-2\le \sf<\d$ to get a nonnegative quantity, then $\Ec_N(f_N,\mu) = \bEc^\tau(\bar{f}_N,\bmu)$.
\end{remark}

\begin{proof}[Proof of \cref{lem:MFEss}]
We already computed the invariance of the relative entropy with \cref{lem:REss}, so it only remains to consider the modulated energy. Using the homogeneity of $\g$ if $\sf\ne 0$ or the product-to-sum property of $\g$ if $\sf=0$, we compute
\begin{align}
\Fr_N(X_N, \mu) &= \frac{1}{2N^2}\sum_{1\leq i\neq j\leq N}\g(x_i-x_j) -\frac{1}{N}\sum_{i=1}^N \int_{\R^\d}\g(x_i-y)\mu(y)dy \nn\\
&\ph + \frac12\int_{(\R^\d)^2}\g(x-y)\mu(x)\mu(y)dxdy \nn\\
&= e^{-\sf\tau/2}\Bigg(\frac{1}{2N^2}\sum_{1\leq i\neq j\leq N}\g(e^{-\tau/2}({x_i-x_j})) -\frac{1}{N}\sum_{i=1}^N \int_{\R^\d}\g(e^{-\tau/2}({x_i-y}))e^{-\d\tau/2}\bmu(e^{-\tau/2}{y})dy \nn\\
&\ph + \frac12\int_{(\R^\d)^2}\g(e^{-\tau/2}({x-y}))e^{-\d\tau}\bmu(e^{-\tau/2}{x})\bmu(e^{-\tau/2}{y})dydx\Bigg)  +\frac{1}{2N}\log(e^{\tau/2})\indic_{\sf=0}.
\end{align}
Therefore, making the changes of variable  $\xi_i = e^{-\tau/2}{x_i}$, $\xi= e^{-\tau/2}{x}$, $\eta= e^{-\tau/2}{y}$, we find that
\begin{multline}
\Fr_N(\XN,\mu)  =  e^{-\sf\tau/2}\Bigg(\frac{1}{2N^2}\sum_{1\le i\ne j\le N}\g(\xi_i-\xi_j) - \frac1N\sum_{i=1}^N \int_{\R^\d}\g(\xi_i-\eta)d\bmu(\eta)\\
+\frac12\int_{(\R^\d)^2}\g(\xi-\eta)d\bmu^{\otimes 2}(\xi,\eta)\Bigg) + \frac{\tau}{4N}\indic_{\sf=0} = e^{-\sf\tau/2}\Fr_N(\Xi_N,\bmu) + \frac{\tau}{4N}\indic_{\sf=0} = \bar{F}_N^\tau(\Xi_N,\bmu) + \frac{\tau}{4N}\indic_{\sf=0} ,
\end{multline}
where we use the notation $\bFr^\tau(\Xi_N, \bmu)$ from \eqref{eq:bFNtaudef}. Taking the expectation of the left-most side with respect to $f_N(\XN)$ and of the right-most side with respect to $\bar{f}_N(\Xi_N)$ yields \eqref{eq:MFEss1}. Combining \eqref{eq:MFEss1} with \eqref{eq:REss} yields \eqref{eq:MFEss2}.

Next, observe that {since $\|\mu\|_{L^\infty} = e^{-\d\tau/2}\|\bmu\|_{L^\infty}$,}
\begin{align}\label{eq:EcbEc}
\Ec_N(f_N,\mu)&=\frac1\be H_N(\bar{f}_N \vert \bmu^{\otimes N}) + \E_{\bar{f}_N}\Big[ \bFr^\tau(\Xi_N,\bmu)\Big]+ \frac{\tau}{4N}\indic_{\sf=0} \nn\\
&\ph+ \mathsf{C}_1\frac{\log(N e^{-\d\tau/2}\|\bmu\|_{L^\infty})}{N}\indic_{\sf =0} + \mathsf{C}_2 e^{-\ga\d\tau/2} \|\bmu\|_{L^\infty}^{\ga} N^{\la} \nn\\
&= \frac1\be H_N(\bar{f}_N \vert \bmu^{\otimes N}) +\E_{\bar{f}_N}\Big[ \bFr^\tau(\Xi_N,\bmu)\Big]+\Big(\mathsf{C}_1\frac{\log(N\|\bmu\|_{L^\infty})}{ N}  + \frac{\tau}{N}\Big(\frac14-\frac{\d\mathsf{C}_1}{2}\Big)\Big)\indic_{\sf=0} \nn\\
&\ph + \mathsf{C}_2e^{-\ga\d\tau/2}\|\bmu^\tau\|_{L^\infty}^{\ga}N^{\la}.
\end{align}
This gives \eqref{eq:MFEss3}.
\end{proof}

\section{Main proof}\label{sec:MAIN}
In this section, we prove our main result \cref{thm:main}. We divide into two subsections. \cref{ssec:MAINre} covers the relative entropy portion of the theorem, while \cref{ssec:MAINmfe} covers the modulated free energy portion.

\subsection{Relative entropy}\label{ssec:MAINre}
We recall a functional inequality, originally due to \cite{JW2018}, for estimating commutator terms by the relative entropy in the $\sf=0$ case. See also \cite[Section 5]{LLN2020} for a simpler proof.

\begin{lemma}\label{lem:REcomm}
Let $\k: (\R^\d)^2\setminus\triangle \rightarrow\R^\d$ be $C^1$, such that there exists $C_{\k}>0$ for which
\begin{align}
\left|\k(x,y)\right| \leq \frac{C_{\k}}{|x-y|}.
\end{align}
There exists a constant $C_*>0$, depending only on $\d$, such that for any probability measures $\mathbb{P}_N \in \P((\R^\d)^N)$, $\nu\in\P(\R^\d)$, and Lipschitz vector field $v:\R^\d\rightarrow\R^\d$, 
\begin{multline}
\E_{\mathbb{P}_N}\left[\int_{(\R^\d)^2\setminus\triangle} (v(x)-v(y))\cdot\k(x,y) d\Big(\frac1N\sum_{i=1}^N \delta_{x_i}-\nu \Big)^{\otimes 2}(x,y) \right]\\
 \leq {C_*} C_{\k}\| \nabla v\|_{L^\infty}\Big(\frac1N +  H_N(\mathbb{P}_N \vert \nu^{\otimes N})\Big).
\end{multline}
\end{lemma}

For a solution $\bar{f}_N^\tau$ of the self-similar forward Kolmogorov equation \eqref{eq:FKss} and a solution $\bmu^\tau$ of the self-similar mean-field equation \eqref{eq:MFlimss}, we compute the evolution of the normalized relative entropy $H_N(\bar{f}_N^\tau \vert (\bmu^\tau)^{\otimes N})$ introduced in \eqref{eq:RE}.

\begin{lemma}\label{lem:REdiss}
Let $\bar{f}_N$ and $\bmu$ be solutions of equations \eqref{eq:FKss} and \eqref{eq:MFlimss}, respectively, on $[0,\tau_0]$. Then for $\tau \in [0,\tau_0]$,
\begin{multline}\label{eq:dtHNfin}
\frac{d}{d\tau}H_N(\bar{f}_N^{\tau} \vert (\bmu^{\tau})^{\otimes N}) = -\frac{1}{\beta}I_N(\bar{f}_N^{\tau} \vert (\bmu^{\tau})^{\otimes N}) \\
+\E_{\bar{f}_N^\tau}\Big[\frac1N\sum_{i=1}^N \PV \ \k_\tau\ast \Big(\frac1N\sum_{j=1}^N \delta_{\xi_j} - \bmu^\tau\Big)(\xi_i) \cdot \nabla_{\xi_i}\log\left(\frac{\bar{f}_N^{\tau}}{(\bmu^{\tau})^{\otimes N}}\right)\Big]\\
\end{multline}
where
\begin{multline}\label{eq:dtHNfin'}
\frac1N\E_{\bar{f}_N^\tau}\Big[\sum_{i=1}^N  \PV \ \k_\tau\ast \Big(\frac1N\sum_{j=1}^N \delta_{\xi_j} - \bmu^\tau\Big)(\xi_i) \cdot \nabla_{\xi_i}\log\left(\frac{\bar{f}_N^{\tau}}{(\bmu^{\tau})^{\otimes N}}\right)\Big] \\
=-\frac12\E_{\bar{f}_N^\tau}\Big[\int_{(\R^\d)^2\setminus\triangle}\Big(\nabla\log(\bmu^\tau)(\xi)-\nabla\log(\bmu^\tau)(\eta)\Big)\cdot\k_\tau(\xi-\eta)\Big(\frac1N\sum_{j=1}^N\delta_{\xi_j} - \bmu^{\tau}\Big)^{\otimes 2}(\xi,\eta)\Big]\\
  - \E_{\bar{f}_N^\tau}\Big[\frac1N\sum_{i=1}^N\int_{(\R^\d)^2\setminus\triangle}  \div\k_\tau(\xi-\eta)\Big(\frac1N\sum_{j=1}^N\delta_{\xi_j} - \bmu^{\tau}\Big)^{\otimes 2}(\xi,\eta)\Big].
\end{multline}
\end{lemma}
\begin{proof}
We argue formally, leaving to the reader the necessary approximations to deal with questions of regularity.

By the chain rule,
\begin{multline}
\frac{d}{d\tau}H_N(\bar{f}_N^{\tau} \vert (\bmu^{\tau})^{\otimes N}) = \frac1N\int_{(\R^\d)^N}\p_\tau\bar{f}_N^{\tau}\log\left(\frac{\bar{f}_N^{\tau}}{(\bmu^{\tau})^{\otimes N}}\right)d\Xi_N + \frac1N\int_{(\R^\d)^N}\p_{\tau}\bar{f}_N^{\tau} d\Xi_N \\
-\frac1N\int_{(\R^\d)^N}\p_{\tau}\Big(\log(\bmu^{\tau})^{\otimes N}\Big)d\bar{f}_N^{\tau}.
\end{multline}
The second term is zero by fundamental theorem of calculus, since $\p_{\tau} \bar{f}_N^{\tau}$ is in divergence form. Substituting in equation \eqref{eq:FKss} for $\p_\tau\bar{f}_N^\tau$,
\begin{multline}
\frac1N\int_{(\R^\d)^N}\p_\tau\bar{f}_N^{\tau}\log\left(\frac{\bar{f}_N^{\tau}}{(\bmu^{\tau})^{\otimes N}}\right)d\Xi_N = \frac1N\sum_{i=1}^N \int_{(\R^\d)^N}\L_{\be,\xi_i}^*\bar{f}_N^\tau\log\left(\frac{\bar{f}_N^\tau}{(\bmu^{\tau})^{\otimes N}}\right)d\Xi_N  \\
-\frac1N\sum_{i=1}^N \int_{(\R^\d)^N}\div_{\xi_i}\left(\bar{f}_N^\tau \frac1N\sum_{1\leq j\leq N : j\neq i}\k_\tau(\xi_i-\xi_j)\right)\log\left(\frac{\bar{f}_N^\tau}{(\bmu^{\tau})^{\otimes N}}\right)d\Xi_N.
\end{multline}
Abbreviating $\bmu_\be^\infty = \bmu_\be$, we see from \eqref{eq:MFlimss}, the chain rule, and the log product-to-sum property that 
\begin{align}
\p_{\tau}\Big(\log(\bmu^{\tau})^{\otimes N}\Big)(\Xi_N) &= \sum_{i=1}^N\left[ \frac1\beta\Delta\log\frac{\bmu^{\tau}}{\bmu_\beta} + \nabla\log{\bmu^{\tau}}\cdot \left(\frac1\beta\nabla\log\frac{\bmu^{\tau}}{\bmu_\beta} - \k_\tau\ast\bmu^{\tau}\right) - \div\k_\tau\ast\bmu^\tau\right](\xi_i) \nn\\
&= \frac{1}{\beta}\Delta_{\Xi_N}\log\left(\frac{\bmu^{\tau}}{\bmu_\beta}\right)^{\otimes N}(\Xi_N) +  \nabla_{\Xi_N}\log(\bmu^{\tau})^{\otimes N}\cdot \frac1\beta\nabla_{\Xi_N}\log\left(\frac{\bmu^{\tau}}{\bmu_\beta}\right)^{\otimes N} \nn\\
&\ph\qquad-\sum_{i=1}^N\Big[ \nabla\log \bmu^\tau(\xi_i)\cdot \k_\tau \ast\bmu^{\tau}(\xi_i) + \div\k_{\tau}\ast\bmu^\tau(\xi_i)\Big],
\end{align}
so that
\begin{multline}
-\int_{(\R^\d)^N}\p_{\tau}\Big(\log(\bmu^{\tau})^{\otimes N}\Big) d\bar{f}_N^{\tau} = -\frac1\be\int_{(\R^\d)^N}\Big[\Delta_{\Xi_N}\log\left(\frac{\bmu^{\tau}}{\bmu_\beta}\right)^{\otimes N} + \nabla_{\Xi_N}\log(\bmu^{\tau})^{\otimes N}\cdot \nabla_{\Xi_N}\log(\frac{\bmu^{\tau}}{\bmu_\be})^{\otimes N}\Big] d\bar{f}_N^{\tau}\\
+\sum_{i=1}^N \int_{(\R^\d)^N}\Big[\nabla\log\bmu^{\tau}\cdot \k_\tau\ast \bmu^{\tau} + \div\k_\tau \ast \bmu^{\tau}\Big](\xi_i) d\bar{f}_N^\tau.
\end{multline}
Using the identity
\begin{align}
\L_{\be,\xi_i}^* = {\frac1\be}\div_{\xi_i}\left(\bmu_\beta^{\otimes N}\nabla_{\xi_i}\left(\frac{\cdot}{\bmu_\beta^{\otimes N}}\right)\right)
\end{align}
and integration by parts, we see that
\begin{align}
\sum_{i=1}^N\int_{(\R^\d)^N}\L_{\be,\xi_i}^*\bar{f}_N^{\tau} \log\left(\frac{\bar{f}_N^{\tau}}{(\bmu^{\tau})^{\otimes N}}\right)d\Xi_N &= -{\frac1\be}\sum_{i=1}^N\int_{(\R^\d)^N} \bmu_\beta^{\otimes N}\nabla_{\xi_i}\left(\frac{\bar{f}_N^{\tau}}{\bmu_\beta^{\otimes N}}\right)\cdot\nabla_{\xi_i}\log\left(\frac{\bar{f}_N^{\tau}}{(\bmu^{\tau})^{\otimes N}}\right)d\Xi_N \nn\\
&=-{\frac1\be}\int_{(\R^\d)^N} \nabla_{\Xi_N}\log\left(\frac{\bar{f}_N^{\tau}}{\bmu_\beta^{\otimes N}}\right)\cdot\nabla_{\Xi_N}\log\left(\frac{\bar{f}_N^{\tau}}{(\bmu^{\tau})^{\otimes N}}\right)d\bar{f}_N^{\tau}. \label{eq:REL1}
\end{align}
Similarly,
\begin{align}
&-\frac1\be\int_{(\R^\d)^N}\Big[\Delta_{\Xi_N}\log\left(\frac{\bmu^{\tau}}{\bmu_\beta}\right)^{\otimes N} + \nabla_{\Xi_N}\log(\bmu^{\tau})^{\otimes N}\cdot \nabla_{\Xi_N}\log\left(\frac{\bmu^{\tau}}{\bmu_\beta}\right)^{\otimes N}\Big] d\bar{f}_N^{\tau}\nn\\
&= \frac1\beta\int_{(\R^\d)^N}\nabla_{\Xi_N}\log\left(\frac{\bmu^{\tau}}{\bmu_\beta}\right)^{\otimes N}\cdot\nabla_{\Xi_N}\log\bar{f}_N^{\tau} d\bar{f}_N^{\tau} \nn\\
&\ph- \frac1\beta\int_{(\R^\d)^N}\nabla_{\Xi_N}\log(\bmu^{\tau})^{\otimes N}\cdot\nabla_{\Xi_N}\log\left(\frac{\bmu^{\tau}}{\bmu_\beta}\right)^{\otimes N}d\bar{f}_N^{\tau}\nn\\
&=\frac{1}{\beta}\int_{(\R^\d)^N}\nabla_{\Xi_N}\log\left(\frac{\bmu^{\tau}}{\bmu_\beta}\right)^{\otimes N}\cdot \nabla_{\Xi_N}\log(\frac{\bar{f}_N^{\tau}}{(\bmu^{\tau})^{\otimes N}}) d\bar{f}_N^{\tau}. \label{eq:REL2}
\end{align}
Combining \eqref{eq:REL1} and \eqref{eq:REL2},
\begin{multline}
 -\frac1{\beta N}\int_{(\R^\d)^N}\left|\nabla_{\Xi_N}\log\left(\frac{\bar{f}_N^{\tau}}{(\bmu^{\tau})^{\otimes N}}\right)\right|^2 d\bar{f}_N^{\tau} = \frac1N\sum_{i=1}^N\int_{(\R^\d)^N}\L_{\be, \xi_i}^*\bar{f}_N^{\tau} \log\left(\frac{\bar{f}_N^{\tau}}{(\bmu^{\tau})^{\otimes N}}\right)d\Xi_N\\
-\frac{1}{\beta N}\int_{(\R^\d)^N}\Big[\Delta_{\Xi_N}\log\left(\frac{\bmu^{\tau}}{\bmu_\beta}\right)^{\otimes N} + \nabla_{\Xi_N}\log(\bmu^{\tau})^{\otimes N}\cdot\nabla_{\Xi_N}\log\left(\frac{\bmu^{\tau}}{\bmu_\beta}\right)^{\otimes N}\Big] d\bar{f}_N^{\tau}.\label{eq:dtHNI}
\end{multline}

Next, integrating by parts the first two terms and recombining, using that $\nabla_{\xi_i}\bar{f}_N^\tau = \bar{f}_N^\tau\nabla_{\xi_i}\log\bar{f}_N^\tau$,
\begin{multline}\label{eq:dtHNfinII}
-\frac1N\sum_{i=1}^N \int_{(\R^\d)^N}\div_{\xi_i}\left(\bar{f}_N^\tau \frac1N\sum_{1\leq j\leq N : j\neq i}\k_\tau(\xi_i-\xi_j)\right)\log\left(\frac{\bar{f}_N^{\tau}}{(\bmu^{\tau})^{\otimes N}}\right)d\Xi_N\\
 + \frac1N\sum_{i=1}^N \int_{(\R^\d)^N} \div(\k_\tau \ast \bmu^{\tau})(\xi_i) d\bar{f}_N^\tau +   \frac1N\sum_{i=1}^N \int_{(\R^\d)^N} \nabla\log\bmu^{\tau}(\xi_i)\cdot (\k_\tau\ast \bmu^{\tau})(\xi_i) d\bar{f}_N^\tau\\
= \frac1N\sum_{i=1}^N \int_{(\R^\d)^N} \PV \ \k_\tau\ast \Big(\frac1N\sum_{j=1}^N \delta_{\xi_j} - \bmu^\tau\Big)(\xi_i) \cdot \nabla_{\xi_i}\log\left(\frac{\bar{f}_N^{\tau}}{(\bmu^{\tau})^{\otimes N}}\right) d\bar{f}_N^\tau,
\end{multline}
where $\PV$ denotes principal value. On the other hand, writing $\log\frac{\bar{f}_N^\tau}{(\bmu^\tau)^{\otimes N}}$ as a difference of logs and integrating by parts the term with $\nabla_{\xi_i}\log\bar{f}_N^\tau$, we see that
\begin{multline}\label{eq:dtHNfinII'}
\frac1N\sum_{i=1}^N \int_{(\R^\d)^N} \PV \ \k_\tau\ast \Big(\frac1N\sum_{j=1}^N \delta_{\xi_j} - \bmu^\tau\Big)(\xi_i) \cdot \nabla_{\xi_i}\log\left(\frac{\bar{f}_N^{\tau}}{(\bmu^{\tau})^{\otimes N}}\right) d\bar{f}_N^\tau\\
=-\frac1N\sum_{i=1}^N\int_{(\R^\d)^N} \nabla\log\bmu^\tau(\xi_i) \cdot\PV \ \k_\tau\ast \Big(\frac1N\sum_{j=1}^N\delta_{\xi_j} - \bmu^{\tau}\Big)(\xi_i) d\bar{f}_N^{\tau} \\ - \frac1N\sum_{i=1}^N\int_{(\R^\d)^N}\PV \ \div\k\ast \Big(\frac1N\sum_{j=1}^N\delta_{\xi_j} - \bmu^{\tau}\Big)(\xi_i)d\bar{f}_N^{\tau}.
\end{multline}
Hence, combining \eqref{eq:dtHNfinII}, \eqref{eq:dtHNfinII'}, we obtain
\begin{multline}\label{eq:dtHNfinII''}
-\frac1N\sum_{i=1}^N \int_{(\R^\d)^N}\div_{\xi_i}\left(\bar{f}_N^\tau \frac1N\sum_{1\leq j\leq N : j\neq i}\k_\tau(\xi_i-\xi_j)\right)\log\left(\frac{\bar{f}_N^{\tau}}{(\bmu^{\tau})^{\otimes N}}\right)d\Xi_N\\
 + \frac1N\sum_{i=1}^N \int_{(\R^\d)^N}  \div(\k_\tau \ast \bmu^{\tau})(\xi_i) d\bar{f}_N^\tau +   \sum_{i=1}^N \int_{(\R^\d)^N} \nabla\log\bmu^{\tau}(\xi_i)\cdot (\k_\tau\ast \bmu^{\tau})(\xi_i) d\bar{f}_N^\tau\\
 =
 -\frac1N\sum_{i=1}^N\int_{(\R^\d)^N} \nabla\log \bmu^{\tau}(\xi_i)\cdot\PV \ \k_\tau\ast \Big(\frac1N\sum_{j=1}^N\delta_{\xi_j} - \bmu^{\tau}\Big)(\xi_i) d\bar{f}_N^{\tau} \\ 
 -\frac1N \sum_{i=1}^N\int_{(\R^\d)^N}\PV \ \div\k_\tau\ast \Big(\frac1N\sum_{j=1}^N\delta_{\xi_j} - \bmu^{\tau}\Big)(\xi_i)d\bar{f}_N^{\tau}.
\end{multline}
Since
\begin{align}
\int_{\R^\d}\bmu^{\tau}\nabla\log\bmu^{\tau}\cdot \PV \ \k_\tau\ast \Big(\frac1N\sum_{j=1}^N\delta_{\xi_j} - \bmu^{\tau}\Big) d\xi &=\int_{\R^\d}\nabla\bmu^{\tau} \PV \ \k_\tau\ast \left(\frac1N\sum_{j=1}^N\delta_{\xi_j} - \bmu^{\tau}\right)  d\xi \nn\\
&=-\int_{\R^\d}\bmu^\tau \PV \ \div\k_\tau\ast\Big(\frac1N\sum_{j=1}^N\delta_{\xi_j} - \bmu^{\tau}\Big)d\xi,
\end{align}
we see, after symmetrizing, that
\begin{multline}\label{eq:dtHNfinIII}
\frac1N\sum_{i=1}^N\int_{(\R^\d)^N} \nabla\log \bmu^{\tau}(\xi_i)\cdot \PV \ \k_\tau \ast \Big(\frac1N\sum_{j=1}^N\delta_{\xi_j} - \bmu^{\tau}\Big)(\xi_i) d\bar{f}_N^{\tau} = \\
\frac12\int_{(\R^\d)^N}\int_{(\R^\d)^2\setminus\triangle}\Big(\nabla\log(\bmu^\tau)(\xi)-\nabla\log(\bmu^\tau)(\eta)\Big)\cdot\k_\tau(\xi-\eta)\Big(\frac1N\sum_{j=1}^N\delta_{\xi_j} - \bmu^{\tau}\Big)^{\otimes 2}(\xi,\eta)d\bar{f}_N^{\tau} \\
-\frac1N\sum_{i=1}^N \int_{(\R^\d)^N}\int_{\R^\d}\div\k_\tau\ast\Big(\frac1N\sum_{j=1}^N\delta_{\xi_j} - \bmu^{\tau}\Big) d\bmu^\tau d\bar{f}_N^{\tau}.
\end{multline}
Inserting this identity into the right-hand side of \eqref{eq:dtHNfinII''}, we find that
\begin{multline}\label{eq:dtHNfinIII'}
-\frac1N\sum_{i=1}^N \int_{(\R^\d)^N}\div_{\xi_i}\left(\bar{f}_N^\tau \frac1N\sum_{1\leq j\leq N : j\neq i}\k_\tau(\xi_i-\xi_j)\right)\log\left(\frac{\bar{f}_N^{\tau}}{(\bmu^{\tau})^{\otimes N}}\right)d\Xi_N\\
 + \frac1N\sum_{i=1}^N \int_{(\R^\d)^N} \div(\k_\tau \ast \bmu^{\tau})(\xi_i) d\bar{f}_N^\tau +   \frac1N\sum_{i=1}^N \int_{(\R^\d)^N} \nabla\log\bmu^{\tau}(\xi_i)\cdot (\k_\tau\ast \bmu^{\tau})(\xi_i) d\bar{f}_N^\tau \\
 =-\frac12\int_{(\R^\d)^N}\int_{(\R^\d)^2\setminus\triangle}\Big(\nabla\log(\bmu^\tau)(\xi)-\nabla\log(\bmu^\tau)(\eta)\Big)\cdot\k_\tau(\xi-\eta)\Big(\frac1N\sum_{j=1}^N\delta_{\xi_j} - \bmu^{\tau}\Big)^{\otimes 2}(\xi,\eta)d\bar{f}_N^{\tau}\\
  - \frac1N\sum_{i=1}^N\int_{(\R^\d)^N}\int_{(\R^\d)^2\setminus\triangle}  \div\k_\tau(\xi-\eta)\Big(\frac1N\sum_{j=1}^N\delta_{\xi_j} - \bmu^{\tau}\Big)^{\otimes 2}(\xi,\eta) d\bar{f}_N^{\tau}.
 \end{multline}

Combining \eqref{eq:dtHNI} and \eqref{eq:dtHNfinIII'}, we arrive at the desired \eqref{eq:dtHNfin}, \eqref{eq:dtHNfin'}.
\end{proof}

We now close the Gr\"{o}nwall loop in the relative entropy estimate for when $\M$ is antisymmetric and $\sf=0$. This then establishes the estimate \eqref{eq:mainRE} from \cref{thm:main}.

\begin{prop}\label{prop:REgron}
Suppose that $\sf=0$ and $\M$ is antisymmetric. Then for $\tau\ge 0$,
\begin{multline}\label{eq:REgron1}
H_N(\bar{f}_N^{\tau} \vert (\bmu^{\tau})^{\otimes N})  \le {\exp\Big({\int_0^{\tau}\Big(\|\nabla^{\otimes 2}\log\frac{\bmu^{\tau'}}{\bmu_\be^\infty}\|_{L^\infty} - e^{-2\|\log\frac{\bmu^{\tau'}}{\bmu_\be^\infty}\|_{L^\infty}}\Big)d\tau'  }\Big)} H_N(\bar{f}_N^{0} \vert (\bmu^{0})^{\otimes N}) \\
+ \frac{C_*}{N}\int_0^\tau {\exp\Big({\int_{\tau'}^{\tau}\Big(\|\nabla^{\otimes 2}\log\frac{\bmu^{\tau''}}{\bmu_\be^\infty}\|_{L^\infty} - e^{-2\|\log\frac{\bmu^{\tau''}}{\bmu_\be^\infty}\|_{L^\infty}}\Big)d\tau''  }\Big)}\| \nabla^{\otimes 2}\log\frac{\bmu^{\tau'}}{\bmu_\be^\infty}\|_{L^\infty}   d\tau',
\end{multline}
and for $t\ge 0$,
\begin{multline}\label{eq:REgron2}
H_N({f}_N^{t} \vert (\mu^{t})^{\otimes N}) \le {\exp\Big({\int_0^{\log(t+1)}\Big(\|\nabla^{\otimes 2}\log\frac{\bmu^{\tau'}}{\bmu_\be^\infty}\|_{L^\infty} - e^{-2\|\log\frac{\bmu^{\tau'}}{\bmu_\be^\infty}\|_{L^\infty}}\Big)d\tau'  }\Big)} H_N(f_N^0 \vert (\mu^0)^{\otimes N})  \\
+ \frac{C_*}{N}\int_0^{\log(t+1)} {\exp\Big({\int_{\tau'}^{\log(t+1)}\Big(\|\nabla^{\otimes 2}\log\frac{\bmu^{\tau''}}{\bmu_\be^\infty}\|_{L^\infty} - e^{-2\|\log\frac{\bmu^{\tau''}}{\bmu_\be^\infty}\|_{L^\infty}}\Big)d\tau''  }\Big)}\| \nabla^{\otimes 2}\log\frac{\bmu^{\tau'}}{\bmu_\be^\infty}\|_{L^\infty}   d\tau',
\end{multline}
where $C_*$ is the constant from \cref{lem:REcomm}.
\end{prop}
\begin{proof}
We continue to abbreviate $\bmu_\be^\infty = \bmu_\be$. Since $\div\k = 0$ and $(x-y)\cdot\k(x-y)=0$, hence
\begin{align}
\Big(\nabla\log\bmu_\be(x) - \nabla\log\bmu_\be(y)\Big)\cdot \k(x-y) = 0,
\end{align}
the dissipation identity \eqref{eq:dtHNfin} simplifies to
\begin{multline}
\frac{d}{d\tau}H_N(\bar{f}_N^{\tau} \vert (\bmu^{\tau})^{\otimes N}) = -\frac{1}{\beta}I_N(\bar{f}_N^{\tau} \vert (\bmu^{\tau})^{\otimes N}) \\
-\frac12\E_{\bar{f}_N^\tau}\Bigg[\int_{(\R^\d)^2\setminus\triangle}\Big(\nabla\log(\frac{\bmu^\tau}{\bmu_\be})(\xi)-\nabla\log(\frac{\bmu^\tau}{\bmu_\be})(\eta)\Big)\cdot\k(\xi-\eta)\Big(\frac1N\sum_{j=1}^N\delta_{\xi_j} - \bmu^{\tau}\Big)^{\otimes 2}(\xi,\eta)\Bigg].
\end{multline}
Since $\bmu_\be$ satisfies an LSI with constant $\frac1\be$ \cite{Gross1975}, Holley-Stroock perturbation implies that $\bmu^\tau$ satisfies an LSI with constant $\frac{1}{\be}e^{2\|\log\frac{\bmu^\tau}{\bmu_\be}\|_{L^\infty}}$, and by tensorization $(\bmu^\tau)^{\otimes N}$ satisfies an LSI with the same constant.\footnote{We learned of this argument from \cite{GlBM2021}.} For instance, see \cite[Propositions 5.6.1, 6.5.1]{BGL2014}. Hence,
\begin{align}\label{eq:FIbfNbmu}
-\frac{1}{\beta}I_N(\bar{f}_N^{\tau} \vert (\bmu^{\tau})^{\otimes N}) \le -e^{-2\|\log\frac{\bmu^\tau}{\bmu_\be}\|_{L^\infty}}H_N(\bar{f}_N^{\tau} \vert (\bmu^{\tau})^{\otimes N}).
\end{align}
Applying \cref{lem:REcomm} to the commutator term ($C_{\k}=1$ in this application), we then find that
\begin{multline}
\frac{d}{d\tau}H_N(\bar{f}_N^{\tau} \vert (\bmu^{\tau})^{\otimes N}) \le - e^{-2\|\log\frac{\bmu^\tau}{\bmu_\be}\|_{L^\infty}}H_N(\bar{f}_N^{\tau} \vert (\bmu^{\tau})^{\otimes N}) \\
+\frac{{C_*}\| \nabla^{\otimes 2}\log\frac{\bmu^\tau}{\bmu_\be}\|_{L^\infty}}{N} + C_*\| \nabla^{\otimes 2}\log\frac{\bmu^\tau}{\bmu_\be}\|_{L^\infty} H_N(\bar{f}_N^\tau \vert (\bmu^\tau)^{\otimes N}).
\end{multline}
Integrating this differential inequality yields \eqref{eq:REgron1}. Appealing to \cref{lem:REss} to return to our original $(t,x)$ coordinates yields \eqref{eq:REgron2}, completing the proof of the proposition.
\end{proof}

\subsection{Modulated free energy}\label{ssec:MAINmfe}

We recall the almost positivity estimate for the modulated energy. The case $\sf<\d-2$ is taken from \cite[Remark 5.10]{RS2021}, which is a refinement of \cite[Proposition 2.1]{NRS2021}; while the second case $\d>\sf\ge \d-2$ is taken from \cite{RS2022} (see also \cite[Proposition 6.4]{CdCRS2023}), which improves upon \cite[Corollary 3.4]{Serfaty2020}. {It is expected that \eqref{eq:MEpossupC}, in fact, holds for all $\min(\d-2,0)\leq \sf<\d$, but this remains open.}

\begin{lemma}\label{lem:MEpos}
If $\d-2\le \sf<\d$, then
\begin{align}\label{eq:MEpossupC}
\Fr_N(\XN,\mu) \ge -\Big( \frac{\log(N\|\mu\|_{L^\infty})}{2N\d}\indic_{\sf=0} + \Cs\|\mu\|_{L^\infty}^{\frac{\sf}{\d}}N^{\frac{\sf}{\d}-1}\Big),
\end{align}
and if $\sf<\d-2$, then
\begin{align}\label{eq:MEpossubC}
\Fr_N(\XN,\mu) \ge -\mathsf{C}\Big(\frac{(\log N + |\log \|\mu\|_{L^\infty}|)}{N}\indic_{\sf=0} + \|\mu\|_{L^\infty}^{\frac{\sf}{\d}} N^{-\frac{2(\d-\sf)}{2(\d-\sf)+\sf(\d+2)}}\Big).
\end{align}
Above, $\mathsf{C}>0$ is a constant depending only on $\d,\sf$. 
\end{lemma}

We next recall the functional inequality  used to estimate the commutator term by the modulated energy, as discussed in the introduction. The case $\d-2\le \sf<\d$ is taken from \cite{RS2022} (see \cite[Proposition 2.13]{CdCRS2023} for a sketch of the proof on the torus) and is sharp, improving upon \cite[Proposition 1.1]{Serfaty2020}. The case $\sf<\d-2$ is taken from \cite[Proposition 5.15]{RS2021}, which is a refinement of \cite[Proposition 4.1]{NRS2021}. The error estimates presented below differ slightly, though not in an essential way, from the cited reference due to a more a precise optimization of the parameters.

\begin{lemma}\label{lem:MEcomm}
There exists a constant $C>0$, depending only on $\d,\sf$, such that for a vector field $v:\R^\d\rightarrow\R^\d$, probability density $\mu$ on $\R^\d$, and pairwise distinct $\XN\in (\R^\d)^N$,
\begin{multline}\label{commutator}
\left|\int_{(\R^\d)^2\setminus\triangle} (v(x)-v(y))\cdot \nabla\g(x-y) d\Big(\frac1N\sum_{i=1}^N\delta_{x_i} - \mu\Big)^{\otimes 2}(x,y)\right|\\ \le C \|v\|_{*} \( F_N(\XN, \mu) +  o_N(1)\),
\end{multline}
where
\begin{align}
\|v\|_{*} \coloneqq \begin{cases} \|\nabla v\|_{L^\infty}, & {\sf\geq \d-2} \\ \|\nabla v\|_{L^\infty} + \|(-\Delta)^{\frac{\d-\sf}{4}}v\|_{L^{\frac{2\d}{\d-2-\sf}}}, & {\sf<\d-2}, \end{cases}
\end{align}
and
\begin{align}\label{eq:oN1def}
o_N(1) = \begin{cases}\frac{\log(N\|\mu\|_{L^\infty})}{2N\d}\indic_{\sf=0} + \Cs\|\mu\|_{L^\infty}^{\frac{\sf}{\d}}N^{\frac{\sf}{\d}-1}, & {\sf\geq\d-2}\\ \\ \frac{\Cs(\log N + |\log\|\mu\|_{L^\infty}|)}{N}\indic_{\sf=0} + \Cs\|\mu\|_{L^\infty}^{\frac{\sf}{\d}}N^{-\frac{{2(\d-\sf)}}{\left(\sf({\d+2})+{2(\d-\sf)}\right)(1+\sf)}}, & {\sf<\d-2}. \end{cases}
\end{align}
\end{lemma}

As advertised in the introduction, the modulated free energy in self-similar coordinates satisfies the following dissipation identity (cf. \eqref{eq:introdtMFE}). The reader will recall the definition of $\bar{E}_N^\tau$ from \eqref{eq:MFEss2}.

\begin{lemma}\label{lem:MFEdiss}
Let $\bar{f}_N$ and $\bmu$ be solutions of equations \eqref{eq:FKss} and \eqref{eq:MFlimss}, respectively, on $[0,\tau_0]$. Then for $\tau \in [0,\tau_0]$, the following hold. If $\M$ is antisymmetric, then
\begin{multline}\label{eq:MFEdiss1}
\frac{d}{d\tau}\bar{E}_N^\tau(\bar{f}_N^\tau,\bmu^\tau) \le -\frac{1}{\be^2}I_N(\bar{f}_N^\tau\vert (\bmu^\tau)^{\otimes N}) -\frac{\sf}{4}\E_{\bar{f}_N^\tau}[\bar{F}_N^\tau(\Xi_N,\bmu^\tau)] - \frac{1}{4N}\indic_{\sf=0}\\
+\frac12\E_{\bar{f}_N^\tau}\Bigg[\int_{(\R^\d)^2\setminus\triangle}\Big(\bar{u}^\tau(\xi)-\bar{u}^\tau(\eta)\Big)\cdot \nabla\g(\xi-\eta)d\Big(\frac1N\sum_{i=1}^N \delta_{\xi_i}-\bmu^\tau\Big)^{\otimes 2}(\xi,\eta)\Bigg] \\
+ \frac{1}{\be}\E_{\bar{f}_N^\tau}\Bigg[\int_{(\R^\d)^2\setminus\triangle} \D\g_\tau(\xi-\eta)d\Big(\frac1N\sum_{i=1}^N \delta_{\xi_i}-\bmu^\tau\Big)^{\otimes 2}(\xi,\eta)\Bigg],
\end{multline}
where $\bar{u}^\tau \coloneqq \M\nabla\log\frac{\bmu^\tau}{\bmu_\be^\tau} + \k_\tau\ast(\bmu^\tau - \bmu_\be^\tau)$. If $\M=-\I$, then
\begin{multline}\label{eq:MFEdiss2}
\frac{d}{d\tau}\bar{E}_N^\tau(\bar{f}_N^\tau,\bmu^\tau) \le -\frac{1}{\be^2}I_N(\bar{f}_N^\tau \vert \mathbb{Q}_{N,\be}^{\tau}(\bmu^\tau)) - \frac{\sf}{2}\E_{\bar{f}_N^\tau}[\bar{F}_N^\tau(\Xi_N,\bmu^\tau)] \\
- \frac12\E_{\bar{f}_N^\tau}\Bigg[\int_{(\R^\d)^2\setminus\triangle}(\bar{u}^\tau(\xi)-\bar{u}^\tau(\eta))\cdot \nabla\g_\tau(\xi-\eta)d\Big(\frac1N\sum_{i=1}^N \delta_{\xi_i}-\bmu^\tau\Big)^{\otimes 2}(\xi,\eta)\Bigg],
\end{multline}
where $\bar{u}^\tau \coloneqq  \frac1\be\nabla\log\frac{\bmu^\tau}{\bmu_\be^\tau} + \nabla\g\ast(\bmu^\tau-\bmu_\be^\tau)$ {and $\Q_{N,\be}(\bmu^\tau)$ is as in \eqref{eq:mGM}.}
\end{lemma}
\begin{proof}
The calculation \eqref{eq:introdtME} from the introduction carries over to establish a dissipation inequality for $\frac{d}{d\tau}\E_{\bar{f}_N^\tau}[\bar{F}_N^\tau(\Xi_N,\bmu^\tau)]$. The vector field $u^t$ is replaced by $\bar{v}^\tau = -\k_\tau\ast\mu^\tau + \frac12\xi$, and there is only the additional term coming from when the derivative $\p_\tau$ hits  $e^{-\frac{\sf\tau}{2}}$. Ultimately, we find that
\begin{multline}\label{eq:dtbFpre}
\frac{d}{d\tau}\E_{\bar{f}_N^\tau}[\bFr^\tau(\Xi_N,\bmu^\tau)] \le \frac{1}{\be}\E_{\bar{f}_N^\tau}\Bigg[\int_{(\R^\d)^2\setminus\triangle} \D\g_\tau(\xi-\eta)d\Big(\frac1N\sum_{i=1}^N \delta_{\xi_i}-\bmu^\tau\Big)^{\otimes 2}(\xi,\eta)\Bigg]\\
+\E_{\bar{f}_N^\tau}\Bigg[\frac1N\sum_{i=1}^N \PV \ \nabla\g_\tau\ast\Big(\frac1N\sum_{i=1}^N \delta_{\xi_i}-\bmu^\tau\Big)(\xi_i) \cdot \PV \ \k_\tau\ast\Big(\frac1N\sum_{i=1}^N \delta_{\xi_i}-\bmu^\tau\Big)(\xi_i)\Bigg]  \\
- \frac12\E_{\bar{f}_N^\tau}\Bigg[\int_{(\R^\d)^2\setminus\triangle}(\bar{v}^\tau(\xi)-\bar{v}^\tau(\eta))\cdot \nabla\g(\xi-\eta)d\Big(\frac1N\sum_{i=1}^N \delta_{\xi_i}-\mu^t\Big)^{\otimes 2}(\xi,\eta)\Bigg] - \frac{\sf}{2}\E_{\bar{f}_N^\tau}[\bar{F}_N^\tau(\Xi_N,\bmu^\tau)].
\end{multline}
Combining this inequality with \eqref{eq:dtHNfin}, \eqref{eq:dtHNfin'} from \cref{lem:REdiss}, we obtain for  antisymmetric  $\M$ ,
\begin{multline}\label{eq:dtENMantispre}
\frac{d}{d\tau}\bar{E}_N^\tau(\bar{f}_N^\tau,\bmu^\tau) \le -\frac{1}{\be^2}I_N(\bar{f}_N^\tau\vert (\bmu^\tau)^{\otimes N}) + \frac{1}{\be}\E_{\bar{f}_N^\tau}\Big[\int_{(\R^\d)^2\setminus\triangle} \D\g_\tau(\xi-\eta)d\Big(\frac1N\sum_{i=1}^N \delta_{\xi_i}-\bmu^\tau\Big)^{\otimes 2}(\xi,\eta)\Big] \\
+\frac12\E_{\bar{f}_N^\tau}\Big[\int_{(\R^\d)^2\setminus\triangle}\Big(\bar{w}^\tau(\xi)-\bar{w}^\tau(\eta)\Big)\cdot \nabla\g(\xi-\eta)d\Big(\frac1N\sum_{i=1}^N \delta_{\xi_i}-\bmu^\tau\Big)^{\otimes 2}(\xi,\eta)\Big]- \frac{\sf}{2}\E_{\bar{f}_N^\tau}[\bar{F}_N^\tau(\Xi_N,\bmu^\tau)] ,
\end{multline}
where $\bar{w}^\tau \coloneqq \frac1\be\M\nabla\log{\bmu^\tau} + \k_\tau\ast\bmu^\tau - \frac12\xi$. Since $\M(\xi-\eta)\cdot(\xi-\eta)=0$, by antisymmetry of $\M$, we see from the relation \eqref{eq:bmubetaurel} that
\begin{multline}
(\bar{w}^\tau(\xi)-\bar{w}^\tau(\eta)) \cdot \nabla\g(\xi-\eta) =  -\frac12(\xi-\eta)\cdot\nabla\g(\xi-\eta) \\
+\Big[\Big(\frac1\be\M\nabla\log\frac{\bmu^\tau}{\bmu_\be^\tau} + \k_\tau\ast(\bmu^\tau-\bmu_\be^\tau)\Big)(\xi)-\Big(\frac1\be\M\nabla\log\frac{\bmu^\tau}{\bmu_\be^\tau} + \k_\tau\ast(\bmu^\tau-\bmu_\be^\tau)\Big)(\eta)\Big]\cdot\nabla\g(\xi-\eta) 
\end{multline}
Note that since {$\nabla\g(\xi-\eta) = -\frac{(\xi-\eta)}{|\xi-\eta|^{\sf+2}}$},
\begin{multline}\label{eq:confdotnabg}
-\frac12 \int_{(\R^\d)^2\setminus\triangle}\Big(\frac12\xi-\frac12\eta\Big)\cdot\nabla\g(\xi-\eta)d\Big(\frac1N\sum_{i=1}^N \delta_{\xi_i}-\mu^t\Big)^{\otimes 2}(\xi,\eta) - \frac{\sf}{2}\E_{\bar{f}_N^\tau}[\bar{F}_N^\tau(\Xi_N,\bmu^\tau)]\\
=\begin{cases}-\frac{\sf}{4}\E_{\bar{f}_N^\tau}[\bar{F}_N^\tau(\Xi_N,\bmu^\tau)], & {\sf>0} \\ -\frac{1}{4N}, & {\sf=0}.  \end{cases}
\end{multline}
where to obtain the $\sf=0$ result, we have used that $\bmu^\tau$ is a probability measure and that there are ${N(N-1)}$ terms in $\int_{(\R^\d)^2\setminus\triangle}d\Big(\sum_{i=1}^N \delta_{\xi_i}\Big)^{\otimes 2}$. Inserting \eqref{eq:confdotnabg} into \eqref{eq:dtENMantispre}, we arrive at \eqref{eq:MFEdiss1}.

In the case $\M=-\I$, we again combine \eqref{eq:dtbFpre} with \eqref{eq:dtHNfin}, \eqref{eq:dtHNfin'} from \cref{lem:REdiss} and observe, since $\k_\tau = -\nabla\g_\tau$, the cancellation of the $\D\g_\tau$ terms to obtain
\begin{multline}\label{eq:dtENsympre}
\frac{d}{d\tau}\bar{E}_N^\tau(\bar{f}_N^\tau,\bmu^\tau) \le -\frac{1}{\be^2}I_N(\bar{f}_N^\tau\vert (\bmu^\tau)^{\otimes N}) - \frac{\sf}{2}\E_{\bar{f}_N^\tau}[\bar{F}_N^\tau(\Xi_N,\bmu^\tau)]\\
- \frac{2}{\be}\E_{\bar{f}_N^\tau}\Bigg[{\frac1N}\sum_{i=1}^N \nabla\g_\tau\ast \Big(\frac1N\sum_{j=1}^N \delta_{\xi_j} - \bmu^\tau\Big)(\xi_i) \cdot \nabla_{\xi_i}\log\left(\frac{\bar{f}_N^{\tau}}{(\bmu^{\tau})^{\otimes N}}\right)\Bigg] \\
-\E_{\bar{f}_N^\tau}\Bigg[\frac1N\sum_{i=1}^N \PV \Big|\nabla\g_\tau\ast\Big(\frac1N\sum_{i=1}^N \delta_{\xi_i}-\bmu^\tau\Big)(\xi_i)\Big|^2\Bigg] \\
- \frac12\E_{\bar{f}_N^\tau}\Bigg[\int_{(\R^\d)^2\setminus\triangle}(\bar{w}^\tau(\xi)-\bar{w}^\tau(\eta))\cdot \nabla\g_\tau(\xi-\eta)d\Big(\frac1N\sum_{i=1}^N \delta_{\xi_i}-\mu^t\Big)^{\otimes 2}(\xi,\eta)\Bigg] ,
\end{multline}
where $\bar{w}^\tau \coloneqq \frac1\be\log\bmu^\tau+ \nabla\g\ast\bmu^\tau + \frac12\xi$. We may recombine terms according to
\begin{multline}\label{eq:recombMGM}
-\frac{1}{\be^2}I_N(\bar{f}_N^\tau\vert (\bmu^\tau)^{\otimes N})  -\E_{\bar{f}_N^\tau}\Bigg[\frac1N\sum_{i=1}^N \PV \Big|\nabla\g_\tau\ast\Big(\frac1N\sum_{i=1}^N \delta_{\xi_i}-\bmu^\tau\Big)\Big|^2(\xi_i)\Bigg]\\
- \frac{2}{\be}\E_{\bar{f}_N^\tau}\Bigg[{\frac1N}\sum_{i=1}^N \g_\tau\ast \Big(\frac1N\sum_{j=1}^N \delta_{\xi_j} - \bmu^\tau\Big)(\xi_i) \cdot \nabla_{\xi_i}\log\left(\frac{\bar{f}_N^{\tau}}{(\bmu^{\tau})^{\otimes N}}\right)\Bigg] \\
=-\frac{1}{\be^2}\E_{\bar{f}_N^\tau}\Bigg[{\frac1N\sum_{i=1}^N}\Big|\nabla_{\xi_i}\log\left(\frac{\bar{f}_N^{\tau}}{(\bmu^{\tau})^{\otimes N}}\right) +\frac{\be}{N}\sum_{j\ne i}\nabla\g_\tau(\xi_i-\xi_j) - \be\nabla\g_\tau\ast\bmu^\tau(\xi_i)\Big|^2\Bigg] =-\frac{1}{\be^2}I_N(\bar{f}_N^\tau \vert \mathbb{Q}_{N,\be}^{\tau}(\bmu^\tau)),
\end{multline}
where $\mathbb{Q}_{N,\be}^{\tau}(\bmu^\tau)$ is the {modulated Gibbs measure \eqref{eq:mGM}.}
Instead of separating out the contribution of the confining potential as in \eqref{eq:confdotnabg}, we keep it packaged in the vector field $\bar{w}^\tau$. Since $\nabla(\frac1\be\log\bmu_\be^\tau + \frac14|\xi|^2+ \g\ast\bmu_\be^\tau) = 0$ by the relation \eqref{eq:bmubetaurel}, the vector field $\bar{w}^\tau$ in fact satisfies
\begin{align}\label{eq:bwbusym}
\bar{w}^\tau = \frac1\be\nabla\log\frac{\bmu^\tau}{\bmu_\be^\tau} + \nabla\g\ast(\bmu^\tau-\bmu_\be^\tau) \eqqcolon \bar{u}^\tau.
\end{align}
Inserting \eqref{eq:recombMGM} and \eqref{eq:bwbusym} into the right-hand side of \eqref{eq:dtENsympre} finally yields the desired \eqref{eq:MFEdiss2}.
\end{proof}


We now close the Gr\"onwall loop for the modulated free energy if $\M$ is antisymmetric and $\sf<\d-2$ or $\M=-\I$ and $0\leq \sf < \d$. As commented in the introduction, $E_N$ does not have a sign, but, using \cref{lem:MEpos}, we can add a correction to it to obtain a nonnegative quantity which is more suited to establishing a Gr\"onwall relation. \cref{prop:MFEgron1} below is a warm-up to the subsequent \cref{prop:MFEgron2}, which is the source of the modulated free energy estimates of \cref{thm:main}.

Before stating \cref{prop:MFEgron1}, we record an important remark that will be used in its proof.

\begin{remark}\label{rem:mon}
If $\mu^t$ is a solution, then for any $1\le p\le \infty$, the quantity $\|\mu^t\|_{L^p}$ is conserved if $\M$ is antisymmetric \cite{CL1995} and decreasing if $\M=-\I$ \cite{BIK2015, RS2021}. Recalling \eqref{eq:bmudef}, this implies that $e^{-\frac{\d\tau}{2}(1-\frac1p)}\|\bmu^\tau\|_{L^p}$ is always decreasing. In particular, $e^{-\frac{\sf\tau}{2}}\|\bmu^\tau\|_{L^\infty}^{\frac{\sf}{\d}}$ and $\log(\|\bmu^\tau\|_{L^\infty}) - \frac{\d\tau}{2}$ are decreasing.
\end{remark}

\begin{prop}\label{prop:MFEgron1}
Define the exponents
\begin{align}\label{eq:al12def}
\al_1 \coloneqq \frac{2(\d-\sf)}{\sf(\d+2) + 2(\d-s)}, \qquad \al_2 \coloneqq \frac{2(\d-\sf-2)}{(\sf+2)(\d+2)+2(\d-\sf-2)}.
\end{align}
For $C_0>0$ is sufficiently large depending only on $\d,\sf$, define
\begin{align}
\bar{\Ec}_N(\bar{f}_N^\tau,\bmu^\tau) \coloneqq \bar{E}_N^\tau(\bar{f}_N^\tau,\bmu^\tau) + \bar{\os}_N^\tau(1), \\
\Ec_N(f_N^t,\mu^t) \coloneqq E_N(f_N^t,\mu^t) + \os_N^t(1),
\end{align}
where\footnote{The choice of $\sup_{\tau'\ge 0}|\log\|\bmu^{\tau'}\|_{L^\infty}|$ is for technical reasons, to deal with the fact that the log changes signs, so $|\log\|\bmu^\tau\|_{L^\infty}|$ could increase by either $\|\bmu^\tau\|_{L^\infty}$ decreasing or increasing.}
\begin{multline}\label{eq:bosNdef}
\bar{\os}_N^\tau(1) \coloneqq \frac{\log(N\|\bmu^\tau\|_{L^\infty})}{2N\d}\indic_{\sf=0} \\
 + \begin{cases}   C_0 e^{-\frac{\sf\tau}{2}}\|\bmu^\tau\|_{L^\infty}^{\frac{\sf}{\d}}N^{\frac{\sf}{\d}-1}, & {\sf\ge \d-2} \\ \\  C_0\frac{(\log N + \sup_{\tau'\ge 0}|\log\|\bmu^{\tau'}\|_{L^\infty}|)}{N}\indic_{\sf=0} + C_0 e^{-\frac{\sf\tau}{2}}\|\bmu^\tau\|_{L^\infty}^{\frac{\sf}{\d}}N^{-\frac{\al_1}{1+\sf}}, & {\sf<\d-2}, \end{cases}
\end{multline}
\begin{multline}\label{eq:osNdef}
\os_N^t(1) \coloneqq \frac{\log(N\|\mu^t\|_{L^\infty})}{2N\d}\indic_{\sf=0} \\
+ \begin{cases}   C_0 \|\mu^t\|_{L^\infty}^{\frac{\sf}{\d}}N^{\frac{\sf}{\d}-1}, & {\sf\ge\d-2} \\ \\  \frac{C_0(\log N + \sup_{t\ge 0} |\log((t+1)^{\frac{\d}{2}}\|\mu^t\|_{L^\infty})|)}{N}\indic_{\sf=0} + C_0\|\mu^t\|_{L^\infty}^{\frac{\sf}{\d}}N^{-\frac{\al_1}{1+\sf}}, & {\sf<\d-2}. \end{cases}
\end{multline} 

If $\M$ is antisymmetric and $\sf<\d-2$, then for $\tau\ge 0$,
\begin{multline}\label{eq:MFEgron11}
\bEc^\tau(\bar{f}_N^\tau , \bmu^\tau) \le e^{C\int_0^\tau \|\bar{u}^{\tau'}\|_*d\tau'}\Bigg[\bEc^0(\bar{f}_N^0 , \bmu^0) +C[1-e^{-\frac{\sf\tau}{2}}]\sup_{[0,\tau]}\|\bmu^{\tau'}\|_{L^\infty}^{\frac{\sf}{\d}} N^{-\al_1}\\
+ \frac{{C}[1-e^{-\frac{\sf\tau}{2}}]}{\be \sf} \sup_{[0,\tau]}\|\bmu^{\tau'}\|_{L^\infty}^{\frac{\sf+2}{\d}}\Big(N^{\frac{\sf+2}{\d}-1}\indic_{\d-4\le \sf <\d-2} + N^{-\al_2}\indic_{\sf<\d-4}\Big)\Bigg]
\end{multline}
and for $t\ge 0$,
\begin{multline}\label{eq:MFEgron11'}
\Ec_N(f_N^t,\mu^t) \le e^{C\int_0^{\log(t+1)} \|\bar{u}^{\tau'}\|_*d\tau'}\Bigg[\Ec_N(f_N^0,\mu^0) + C[1-(t+1)^{-\frac{\sf}{2}}] \sup_{[0,\log(t+1)]} \|\bmu^{\tau'}\|_{L^\infty}^{\frac{\sf}{\d}}N^{-\al_1} \\
+\frac{{C}[1-(t+1)^{-\frac{\sf}{2}}]}{\be\sf}\sup_{[0,\log(t+1)]} \|\bmu^{\tau'}\|_{L^\infty}^{\frac{\sf+2}{\d}}\Big(N^{\frac{\sf+2}{\d}-1}\indic_{\d-4\le \sf <\d-2} + N^{-\al_2}\indic_{\sf<\d-4}\Big)\Bigg],
\end{multline}
where $C>0$ depends only on $\d,\sf$ and we adopt the conventions  $\frac{2[1-e^{-\frac{\sf\tau}{2}}]}{\sf} \coloneqq \tau$, $\frac{2[1-(t+1)^{-\frac{\sf}{2}}]}{\sf}\coloneqq t$ for $\sf=0$. If $\M=-\I$, then for $\tau\ge 0$,
\begin{multline}\label{eq:MFEgron12}
\bEc^\tau(\bar{f}_N^\tau , \bmu^\tau)  \le e^{C\int_0^\tau \|\bar{u}^{\tau'}\|_{*}d\tau'}\Bigg[\bEc^0(\bar{f}_N^0 , \bmu^0) +\frac{1}{4N}\indic_{\sf=0}\\
 + C[1-e^{-\frac{\sf\tau}{2}}]\sup_{[0,\tau]}\|\bmu^{\tau'}\|_{L^\infty}^{\frac{\sf}{\d}}\Big(N^{\frac{\sf}{\d}-1}\indic_{\sf\ge\d-2} +  N^{-\al_1}\indic_{\sf<\d-2}\Big)\Bigg]
\end{multline}
and for $t\ge 0$,
\begin{multline}\label{eq:MFEgron12'}
\Ec_N(f_N^t,\mu^t) \le e^{C\int_0^{\log(t+1)} \|\bar{u}^{\tau'}\|_*d\tau'}\Bigg[\Ec_N(f_N^0,\mu^0)+\frac{1}{4N}\indic_{\sf=0} \\+ C[1-(t+1)^{-\frac{\sf}{2}}] \sup_{[0,\log(t+1)]} \|\bmu^{\tau'}\|_{L^\infty}^{\frac{\sf}{\d}}\Big(N^{\frac{\sf}{\d}-1}\indic_{\sf\ge\d-2} +  N^{-\al_1}\indic_{\sf<\d-2}\Big)\Bigg].
\end{multline}

\end{prop}
\begin{proof}
To lighten the notation, we abbreviate $\bar{\os}_N^\tau(1)=\bar{\os}_N^\tau$. We start with the case of antisymmetric $\M$, for which we restrict to $\sf<\d-2$. Using \eqref{eq:MFEdiss1} from \cref{lem:MFEdiss}, we see that
\begin{multline}\label{eq:MFEantispre}
\frac{d}{d\tau}\bEc^\tau(\bar{f}_N^\tau , \bmu^\tau) \le -\frac{1}{\be^2}I_N(\bar{f}_N^\tau\vert (\bmu^\tau)^{\otimes N}) + \frac{1}{\be}\E_{\bar{f}_N^\tau}\Bigg[\int_{(\R^\d)^2\setminus\triangle} \D\g_\tau(\xi-\eta)d\Big(\frac1N\sum_{i=1}^N \delta_{\xi_i}-\bmu^\tau\Big)^{\otimes 2}(\xi,\eta)\Bigg]\\
+\frac12\E_{\bar{f}_N^\tau}\Bigg[\int_{(\R^\d)^2\setminus\triangle}\Big(\bar{u}^\tau(\xi)-\bar{u}^\tau(\eta)\Big)\cdot \nabla\g(\xi-\eta)d\Big(\frac1N\sum_{i=1}^N \delta_{\xi_i}-\bmu^\tau\Big)^{\otimes 2}(\xi,\eta)\Bigg] \\
-\frac{\sf}{4}\E_{\bar{f}_N^\tau}[\bar{F}_N^\tau(\Xi_N,\bmu^\tau)] -\frac{1}{4N}\indic_{\sf=0}+\frac{d}{d\tau}(\bar{\os}_N^\tau),
\end{multline}
where $\bar{u}^\tau = \M\nabla\log\frac{\bmu^\tau}{\bmu_\be^\tau} + \k_\tau\ast(\bmu^\tau - \bmu_\be^\tau)$. We discard the relative Fisher information term. By \cref{rem:mon},
\begin{align}
\frac{d}{d\tau}(\bar{\os}_N^\tau)\le \frac{1}{4N}\indic_{\sf=0}.
\end{align}
Since $\D\g = {-(\sf+2)(\d-\sf-2)\tl{\g}}$, where $\tl{\g}(\xi) = {\frac{1}{\sf+2}}|\xi|^{-\sf-2}$, we may apply \cref{lem:MEpos} to estimate
\begin{multline}\label{eq:MFEDeltagterm}
\frac{1}{\be}\E_{\bar{f}_N^\tau}\Bigg[\int_{(\R^\d)^2\setminus\triangle} \D\g_\tau(\xi-\eta)d\Big(\frac1N\sum_{i=1}^N \delta_{\xi_i}-\bmu^\tau\Big)^{\otimes 2}(\xi,\eta)\Bigg] \\
\leq \mathsf{C}{(\sf + 2)}(\d-\sf-2)\frac{e^{-\sf\tau/2}}{\be} \|\bmu^\tau\|_{L^\infty}^{\frac{\sf+2}{\d}}\begin{cases}  N^{\frac{\sf+2}{\d}-1}, & {\d-2>\sf\geq\d-4} \\     N^{-\al_2}, &{\sf<\d-4}.\end{cases}
\end{multline}
Similarly, using \cref{lem:MEpos} again, we have
\begin{align}\label{eq:MFEMEterm}
-\frac{\sf}{4}\E_{\bar{f}_N^\tau}\Big[\bFr^\tau(\Xi,\bmu^\tau)] \le \frac{\sf e^{-\sf\tau/2}}{4}\Cs\|\bmu^\tau\|_{L^\infty}^{\frac{\sf}{\d}}\begin{cases} N^{\frac{\sf}{\d}-1}, & {\sf\geq\d-2} \\  N^{-\al_1}, & {\sf<\d-2}.\end{cases}
\end{align}
Finally, we handle the commutator term by applying, pointwise in $\tau$ and $\Xi_N$, \cref{lem:MEcomm}, then taking $\E_{\bar{f}_N^\tau}[\cdot]$ of the resulting expression, leading to 
\begin{align}
&\E_{\bar{f}_N^\tau}\Bigg[\Big|\int_{(\R^\d)^2\setminus\triangle}(u^\tau(\xi)-u^\tau(\eta))\cdot\nabla\g_\tau(\xi-\eta)d\Big(\frac1N\sum_{i=1}^N \delta_{\xi_i}-\bmu^\tau\Big)^{\otimes 2}(\xi,\eta)\Big|\Bigg] \nn\\
&\le C\|\bar{u}^\tau\|_{*}\Bigg(\E_{\bar{f}_N^\tau}\Big[\bFr^\tau(\Xi_N,\bmu^\tau)\Big] +\bar{\os}_N^\tau\Bigg) \nn\\
&\le C\|\bar{u}^\tau\|_{*}\bEc^\tau(\bar{f}_N^\tau,\bmu^\tau),\label{eq:MFECommterm}
\end{align}
where the final line follows from the nonnegativity of the relative entropy.

Combining \eqref{eq:MFEDeltagterm}, \eqref{eq:MFEMEterm}, \eqref{eq:MFECommterm}, we arrive at the differential inequality
\begin{multline}
\frac{d}{d\tau}\bEc^\tau(\bar{f}_N^\tau , \bmu^\tau) \le C\|\bar{u}^\tau\|_{*}\bEc^\tau(\bar{f}_N^\tau, \bmu^\tau)+\frac{\Cs\sf e^{-\sf\tau/2}}{4} \|\bmu^\tau\|_{L^\infty}^{\frac{\sf}{\d}}  N^{-\al_1} \\
+ \mathsf{C}{(\sf+2)}(\d-\sf-2)\frac{e^{-\sf\tau/2}}{\be}\|\bmu^\tau\|_{L^\infty}^{\frac{\sf+2}{\d}}\begin{cases}   N^{\frac{\sf+2}{\d}-1}, & {\d-2>\sf\geq\d-4} \\    N^{-\al_2}, &{\sf<\d-4}.\end{cases}
\end{multline}
Integrating the preceding differential inequality, we find
\begin{multline}
\bEc^\tau(\bar{f}_N^\tau , \bmu^\tau) \le e^{C\int_0^\tau \|\bar{u}^{\tau'}\|_*d\tau'}\bEc^0(\bar{f}_N^0 , \bmu^0) +\frac{\sf\Cs}{4} N^{-\al_1} \int_0^\tau  \exp\Big({C\int_{\tau'}^\tau \|\bar{u}^{\tau''}\|_*d\tau'' -\frac{\sf\tau'}{2}}\Big) \|\bmu^\tau\|_{L^\infty}^{\frac{\sf}{\d}} d\tau' \\
+ \frac{\Cs{(\sf+2)}(\d-\sf-2)}{\be}\int_0^\tau \exp\Big({C\int_{\tau'}^\tau \|\bar{u}^{\tau''}\|_*d\tau'' -\frac{\sf\tau'}{2}}\Big)\|\bmu^{\tau'}\|_{L^\infty}^{\frac{\sf+2}{\d}}\Big(N^{\frac{\sf+2}{\d}-1}\indic_{\d-4\le \sf <\d-2} + N^{-\al_2}\indic_{\sf<\d-4}\Big)d\tau'.
\end{multline}
For $\tau'\in[0,\tau]$, we majorize $\int_{\tau'}^{\tau} \|\bar{u}^{\tau''}\|_{*}d\tau'' \le \int_0^\tau  \|\bar{u}^{\tau''}\|_{*}d\tau'' $ and $\|\bmu^{\tau'}\|_{L^\infty} \le \sup_{[0,\tau]}\|\bmu^{\tau'}\|_{L^\infty}$, then evaluate the resulting integrals to obtain
\begin{multline}
\bEc^\tau(\bar{f}_N^\tau , \bmu^\tau) \le e^{C\int_0^\tau \|\bar{u}^{\tau'}\|_*d\tau'}\bEc^0(\bar{f}_N^0 , \bmu^0) +\frac{\Cs[1-e^{-\frac{\sf\tau}{2}}]}{2}e^{C\int_{0}^\tau \|\bar{u}^{\tau''}\|_*d\tau''}\sup_{[0,\tau]}\|\bmu^{\tau'}\|_{L^\infty}^{\frac{\sf}{\d}} N^{-\al_1}\\
+ \frac{2\Cs{(\sf+2)}(\d-\sf-2)[1-e^{-\frac{\sf\tau}{2}}]}{\be \sf}e^{C\int_{0}^\tau \|\bar{u}^{\tau''}\|_*d\tau''} \sup_{[0,\tau]}\|\bmu^{\tau'}\|_{L^\infty}^{\frac{\sf+2}{\d}}\Big(N^{\frac{\sf+2}{\d}-1}\indic_{\d-4\le \sf <\d-2} + N^{-\al_2}\indic_{\sf<\d-4}\Big).
\end{multline}
Above, we adopt the convention that $\frac{2(1-e^{-\frac{\sf\tau}{2}})}{\sf} \coloneqq \tau$ for $\sf=0$. Relabeling constants gives \eqref{eq:MFEgron11}.

Recalling \eqref{eq:bmudef} and \eqref{eq:MFEss3} from \cref{lem:MFEss}, we revert back to the original $(t,x)$ coordinates. For $\tau=\log(t+1)$, we have
\begin{align}
\bar{\Ec}_N^{\tau}(\bar{f}_N^\tau,\bmu^\tau) &= \bar{E}_N^\tau(\bar{f}_N^\tau, \bmu^\tau) + \frac{\log(N\|\bmu^\tau\|_{L^\infty})}{2\d N}\indic_{\sf=0}\nn\\
&\ph+ \frac{C_0(\log N + \sup_{\tau\ge 0} |\log\|\bmu^\tau\|_{L^\infty}|)}{N}\indic_{\sf=0} + C_0e^{-\frac{\sf\tau}{2}}\|\bmu^\tau\|_{L^\infty}^{\frac{\sf}{\d}}N^{-\frac{\al_1}{1+\sf}}
\nn\\
&= E_N(f_N^t,\mu^t)  + \frac{\log(N\|\mu^t\|_{L^\infty})}{2\d N}\nn\\
&\ph + \frac{C_0(\log N + \sup_{t\ge 0} |\log((t+1)^{\frac{\d}{2}}\|\mu^t\|_{L^\infty})|)}{N}\indic_{\sf=0} + C_0\|\mu^t\|_{L^\infty}^{\frac{\sf}{\d}}N^{-\frac{\al_1}{1+\sf}} \nn\\
&=\Ec_N(f_N^t,\mu^t). \label{eq:bEcNEcN}
\end{align}
Observing that $\bEc^0(\bar{f}_N^0,\bmu^0) = \Ec_N(f_N^0,\mu^0)$ finally yields \eqref{eq:MFEgron11'}.

\medskip
Next, we consider the case $\M=-\I$. The analysis is similar to before; our starting point is now, thanks to \eqref{eq:MFEdiss2} from \cref{lem:MFEdiss} and \cref{rem:mon},
\begin{multline}\label{eq:MFEIpre}
\frac{d}{d\tau}\bEc^\tau(\bar{f}_N^\tau , \bmu^\tau)  \le -\frac{1}{\be^2}I_N(\bar{f}_N^\tau\vert {\Q}_{N,\be}^\tau(\bmu^\tau)) -\frac{\sf}{2}\E_{\bar{f}_N^\tau}[\bar{F}_N^\tau(\Xi_N,\bmu^\tau)] +\frac{1}{4N}\indic_{\substack{\sf = 0 \\ }}\\
-\frac12\E_{\bar{f}_N^\tau}\Bigg[\int_{(\R^\d)^2\setminus\triangle}\Big(\bar{u}^\tau(\xi)-\bar{u}^\tau(\eta)\Big)\cdot \nabla\g(\xi-\eta)d\Big(\frac1N\sum_{i=1}^N \delta_{\xi_i}-\bmu^\tau\Big)^{\otimes 2}(\xi,\eta)\Bigg],
\end{multline}
where $\bar{u}^\tau = \frac1\be\nabla\log\frac{\bmu^\tau}{\bmu_\be^\tau} + \nabla\g\ast(\bmu^\tau-\bmu_\be^\tau)$. We again discard the relative Fisher information term. We then use estimates \eqref{eq:MFEMEterm} and \eqref{eq:MFECommterm} to obtain
\begin{align}
\frac{d}{d\tau}\bEc^\tau(\bar{f}_N^\tau , \bmu^\tau)  \le  C\|\bar{u}^\tau\|_{*}\bar{\Ec}_N^\tau(\bar{f}_N^\tau,\bmu^\tau) +\frac{1}{4N}\indic_{\substack{\sf=0 \\ }} +\frac{ \Cs\|\bmu^\tau\|_{L^\infty}^{\frac{\sf}{\d}}\sf e^{-\sf\tau/2}}{4}\begin{cases}  N^{\frac{\sf}{\d}-1}, & {\sf\geq\d-2} \\    N^{-\al_1}, & {\sf<\d-2}.\end{cases} 
\end{align}
Integrating this differential inequality and performing the same majorizations and simplifications as above, we obtain \eqref{eq:MFEgron12}. Reverting to $(t,x)$ coordinates, using \eqref{eq:bEcNEcN}, yields \eqref{eq:MFEgron12'}. This completes the proof of the proposition.

\end{proof}

In closing the Gr\"onwall relation for $\bar{\Ec}_N^\tau(\bar{f}_N^\tau,\bmu^\tau)$, we have discarded the relative Fisher information term for both cases of $\M$. However, as commented in \cref{ssec:introMR}, this term, through an LSI, is the key to obtaining an estimate where the dependence on the initial relative entropy or modulated free energy, which becomes small only if $f_N^0$ is $\mu^0$-chaotic, vanishes as $\tau\rightarrow\infty$. 
The next proposition establishes estimates \eqref{eq:mainMFE1}, \eqref{eq:mainMFE2} of \cref{thm:main}. With them, the proof of \cref{thm:main} is then complete.

\begin{prop}\label{prop:MFEgron2}
Suppose that $\M$ is antisymmetric and $\sf \in (0,\d-2)$. Define
\begin{align}\label{eq:kadef1}
\ka\coloneqq \inf_{\tau\ge 0} \min\Big(e^{-2\|\log\frac{\bmu^\tau}{\bmu_\be^\infty}\|_{L^\infty}}, \frac{\sf}{4}\Big).
\end{align}
Then for $\tau\ge 0$,
\begin{multline}\label{eq:MFEgron21}
\bar{\Ec}_N^\tau(\bar{f}_N^\tau, \bmu^\tau)  \le e^{-\ka\tau + C\int_0^\tau \|\bar{u}^{\tau'}\|_{*}d\tau'}\Bigg[\bar{\Ec}_N^0(\bar{f}_N^0, \bmu^0) +{C_0[1-e^{-\frac{\sf\tau}{4}}]}\sup_{[0,\tau]}\|\bmu^{\tau'}\|_{L^\infty}^{\frac{\sf}{\d}} N^{-\frac{\al_1}{1+\sf}} \\
+ \frac{{C}[1-e^{-\frac{\sf\tau}{4}}]}{\be\sf}\sup_{[0,\tau]}\|\bmu^{\tau'}\|_{L^\infty}^{\frac{\sf+2}{\d}}\Big(N^{\frac{\sf+2}{\d}-1}\indic_{\d-4\le \sf <\d-2} + N^{-\al_2}\indic_{\sf<\d-4}\Big) \Bigg]
\end{multline}
and for $t\ge 0$,
\begin{multline}\label{eq:MFEgron21'}
{\Ec}_N({f}_N^t, \mu^t)  \le  \frac{e^{C\int_0^{\log(t+1)}  \|\bar{u}^{\tau'}\|_{*}d\tau'}}{(t+1)^{\ka}}\Bigg[{\Ec}_N({f}_N^0, \mu^0) + C_0[1-(t+1)^{-\frac{\sf}{4}}]\sup_{[0,\log(t+1)]} \|\bmu^{\tau'}\|_{L^\infty}^{\frac{\sf}{\d}}N^{-\frac{\al_1}{1+\sf}}\\
+\frac{{C}[1-(t+1)^{-\frac{\sf}{4}}]}{\be\sf}\sup_{[0,\log(t+1)]} \|\bmu^{\tau'}\|_{L^\infty}^{\frac{\sf+2}{\d}}\Big(N^{\frac{\sf+2}{\d}-1}\indic_{\d-4\le \sf} + N^{-\al_2}\indic_{\sf<\d-4}\Big) \Bigg],
\end{multline}
where $\al_1,\al_2$ are as in \eqref{eq:al12def} and $C>0$ depends only on $\d,\sf$.

Suppose that $\M=-\I$ and $\sf \in [0,\d)$. Let $C_{LS,N}^\tau \in [0,\infty]$ be the LSI constant for ${\Q}_{N,\be}^\tau(\bmu^\tau)$, and define
\begin{align}
\ka \coloneqq \inf_{\tau\ge 0} \frac{1}{C_{LS,N}^\tau \be}.
\end{align}
Then
\begin{multline}\label{eq:MFEgron22}
\bar{\Ec}_N^\tau(\bar{f}_N^\tau, \bmu^\tau) \le e^{C\int_0^\tau \|\bar{u}^{\tau'}\|_{*}d\tau'}\Bigg[e^{-{\ka}\tau }\Ec_N^0(\bar{f}_N^0, \bmu^0) + \frac{C[1-e^{-\ka\tau}]}{N}\Big[1+\sup_{[0,\tau]}\log_-(N\|\bmu^{\tau'}\|_{L^\infty})\Big]\indic_{\sf=0}\\
+ \frac{C\ka\be e^{-\ka\tau}[e^{(\ka-\frac{\sf}{2})\tau}-1]}{(\ka-\frac{\sf}{2})} \sup_{[0,\tau]}\begin{cases} \frac{\log(N\|\bmu^{\tau'}\|_{L^\infty})}{2\d N}\indic_{\sf=0} + \Cs\|\bmu^{\tau'}\|_{L^\infty}^{\frac{\sf}{\d}}N^{\frac{\sf}{\d}-1}, & {\d-2\le \sf<\d} \\ \frac{\Cs(\log  N + |\log\|\bmu^{\tau'}\|_{L^\infty}|)}{N}\indic_{\sf=0} + \Cs\|\bmu^{\tau'}\|_{L^\infty}^{\frac{\sf}{\d}}N^{-\al_1}, & {\sf<\d-2} \end{cases}\\
+\frac{C\sf e^{-\ka\tau}[e^{(\ka-\frac{\sf}{2})\tau}-1]}{(\ka-\frac{\sf}{2})} \Big(N^{\frac{\sf}{\d}-1}\indic_{\sf\ge \d-2} + N^{-\al_1}\indic_{\sf<\d-2} \Big)\sup_{[0,\tau]}\|\bmu^{\tau'}\|_{L^\infty}^{\frac{\sf}{\d}}\Bigg]
\end{multline}
and
\begin{multline}\label{eq:MFEgron22'}
\Ec_N(f_N^t,\mu^t) \le {e^{C\int_0^{\log(t+1)} \|\bar{u}^{\tau'}\|_{*}d\tau'}}\Bigg[\frac{\Ec_N(f_N^0,\mu^0)}{(t+1)^{\ka}} + \frac{C[1-(t+1)^{-\ka}]}{N}\Big(1+\sup_{[0,\log(t+1)]} \log_{-}(N\|\bmu^{\tau'}\|_{L^\infty})\Big)\indic_{\sf=0} \\
+ \frac{C\ka\be[(t+1)^{\ka-\frac{\sf}{2}}-1]}{(\ka-\frac{\sf}{2})(t+1)^{{\ka}}}\sup_{[0,\log(t+1)]}\begin{cases} \frac{\log(N\|\bmu^{\tau'}\|_{L^\infty})}{2\d N}\indic_{\sf=0} + \Cs\|\bmu^{\tau'}\|_{L^\infty}^{\frac{\sf}{\d}}N^{\frac{\sf}{\d}-1}, & {\d-2\le \sf<\d} \\ \frac{\Cs(\log  N + |\log\|\bmu^{\tau'}\|_{L^\infty}|)}{N}\indic_{\sf=0} + \Cs\|\bmu^{\tau'}\|_{L^\infty}^{\frac{\sf}{\d}}N^{-\al_1}, & {\sf<\d-2} \end{cases} \\
+\frac{C\sf[(t+1)^{\ka-\frac{\sf}{2}}-1]}{(\ka-\frac{\sf}{2})(t+1)^{{\ka}}}\Big(N^{\frac{\sf}{\d}-1}\indic_{\sf\ge \d-2} + N^{-\al_1}\indic_{\sf<\d-2} \Big)\sup_{[0,\log(t+1)]}\|\bmu^{\tau'}\|_{L^\infty}^{\frac{\sf}{\d}}\Bigg],
\end{multline} 
where $C$ is as above.
\end{prop}
\begin{proof}
We start with the case $\M$ antisymmetric and $\sf\in(0,\d-2)$. Abbreviating $\bmu_\be^\infty=\bmu_\be$, we have, by \eqref{eq:FIbfNbmu}, that
\begin{align}
-\frac{1}{\be^2}I_N(\bar{f}_N^\tau\vert (\bmu^\tau)^{\otimes N}) &\le -\frac1\be e^{-2\|\log\frac{\bmu^\tau}{\bmu_\be}\|_{L^\infty}}H_N(\bar{f}_N^\tau \vert (\bmu^\tau)^{\otimes N}) \nn\\
&\le -\frac1\be\min\Big(e^{-2\|\log\frac{\bmu^\tau}{\bmu_\be}\|_{L^\infty}}, \frac{\sf}{4}\Big)H_N(\bar{f}_N^\tau \vert (\bmu^\tau)^{\otimes N}).
\end{align}
Let us abbreviate $\ka^\tau \coloneqq \min\Big(e^{-2\|\log\frac{\bmu^\tau}{\bmu_\be}\|_{L^\infty}}, \frac{\sf}{4}\Big)$. Combining the preceding inequality with the almost positivity bound \eqref{eq:MFEMEterm} and remembering the definition \eqref{eq:bosNdef} of the correction $\bar{\os}_N^\tau$, we find
\begin{align}
&-\frac{1}{\be^2}I_N(\bar{f}_N^\tau\vert (\bmu^\tau)^{\otimes N})  -\frac{\sf}{4}\E_{\bar{f}_N^\tau}[\bar{F}_N^\tau(\Xi_N,\bmu^\tau)] \nn\\
&\le -\frac{\ka^\tau}\be H_N(\bar{f}_N^\tau\vert (\bmu^\tau)^{\otimes N}) - \frac{\sf}{4}\Big(\E_{\bar{f}_N^\tau}[\bar{F}_N^\tau(\Xi_N,\bmu^\tau)] + \bar{\os}_N^\tau\Big) + \frac{\sf}{4}\bar{\os}_N^\tau \nn\\
&\le -\ka^\tau\Ec_N^\tau(\bar{f}_N^\tau,\bmu^\tau) +  \frac{\sf}{4}\bar{\os}_N^\tau,
\end{align}
where we have implicitly used that $\E_{\bar{f}_N^\tau}[\bar{F}_N^\tau(\Xi_N,\bmu^\tau)] + \bar{\os}_N^\tau\ge 0$ by choice of $\bar{\os}_N^\tau$. Applying the preceding bound to the right-hand side of \eqref{eq:MFEantispre} and then proceeding as in the proof of \cref{prop:MFEgron1}, we obtain the differential inequality
\begin{multline}
\frac{d}{d\tau}\Ec_N^\tau(\bar{f}_N^\tau, \bmu^\tau) \le -\ka^\tau\Ec_N^\tau(\bar{f}_N^\tau, \bmu^\tau)+  \frac{\sf}{4}\bar{\os}_N^\tau  +C\|\bar{u}^\tau\|_{*}\bEc^\tau(\bar{f}_N^\tau, \bmu^\tau) \\
+\frac{\Cs\sf e^{-\sf\tau/2}}{4} \|\bmu^\tau\|_{L^\infty}^{\frac{\sf}{\d}}  N^{-\al_1} + \mathsf{C}{(\sf+2)}(\d-\sf-2)\frac{e^{-\sf\tau/2}}{\be}\|\bmu^\tau\|_{L^\infty}^{\frac{\sf+2}{\d}}\begin{cases}   N^{\frac{\sf+2}{\d}-1}, & {\d-2>\sf\geq\d-4} \\    N^{-\al_2}, &{\sf<\d-4},\end{cases}
\end{multline}
which can be integrated, then simplified, to obtain the closed estimate
\begin{multline}\label{eq:bEcIpostGrongen}
\bar{\Ec}_N^\tau(\bar{f}_N^\tau, \bmu^\tau)  \le e^{\int_0^\tau (-\ka^{\tau'}+  C\|\bar{u}^{\tau'}\|_{*})d\tau'}\bar{\Ec}_N^0(\bar{f}_N^0, \bmu^0) 
+   \frac{\sf}{4}\int_0^\tau e^{\int_{\tau'}^\tau (-\ka^{\tau''}+  C\|\bar{u}^{\tau''}\|_{*})d\tau''}\bar{\os}_N^{\tau'}d\tau' \\
+ \frac{\Cs{(\sf+2)}(\d-\sf-2)}{\be}\sup_{[0,\tau]}\|\bmu^{\tau'}\|_{L^\infty}^{\frac{\sf+2}{\d}}\int_0^\tau \exp\Big({\int_{\tau'}^\tau (-\ka^{\tau''}+ C\|\bar{u}^{\tau''}\|_*)d\tau'' -\frac{\sf\tau'}{2}}\Big)\\
 \Big(N^{\frac{\sf+2}{\d}-1}\indic_{\d-4\le \sf <\d-2} + N^{-\al_2}\indic_{\sf<\d-4}\Big)d\tau' \\
 +\frac{\sf\Cs}{4}N^{-\al_1}\sup_{[0,\tau]}\|\bmu^\tau\|_{L^\infty}^{\frac{\sf}{\d}} \int_0^\tau  \exp\Big({\int_{\tau'}^\tau (-\ka^{\tau'} +C\|\bar{u}^{\tau''}\|_*)d\tau'' -\frac{\sf\tau'}{2}}\Big)   d\tau'.
\end{multline}
Set ${\ka}\coloneqq \inf_{\tau\ge 0} \ka^{\tau}$. Since tautologically $\ka \le \frac{\sf}{4}$, it follows that
\begin{align}
e^{-\int_{\tau'}^{\tau}\ka^{\tau''}d\tau''} e^{-\frac{\sf\tau'}{2}} \le e^{-\ka(\tau-\tau') - \frac{\sf\tau'}{2}} \le e^{-\ka\tau} e^{-\frac{\sf\tau'}{4}}.\label{eq:bEcIpostGrongen2}
\end{align}
Applying this bound to the right-hand side of \eqref{eq:bEcIpostGrongen}, then simplifying, we obtain
\begin{multline}
\bar{\Ec}_N^\tau(\bar{f}_N^\tau, \bmu^\tau)  \le e^{-\ka\tau + C\int_0^\tau \|\bar{u}^{\tau'}\|_{*}d\tau'}\bar{\Ec}_N^0(\bar{f}_N^0, \bmu^0)  +{C_0[1-e^{-\frac{\sf\tau}{4}}]}\sup_{[0,\tau]}\|\bmu^{\tau'}\|_{L^\infty}^{\frac{\sf}{\d}} e^{-\ka\tau + C\int_0^\tau \|\bar{u}^{\tau'}\|_{*}d\tau'} N^{-\frac{\al_1}{1+s}} \\
+ \frac{4\Cs{(\sf+2)}(\d-\sf-2)[1-e^{-\frac{\sf\tau}{4}}]}{\be\sf}\sup_{[0,\tau]}\|\bmu^{\tau'}\|_{L^\infty}^{\frac{\sf+2}{\d}}e^{-\ka\tau + C\int_0^\tau \|\bar{u}^{\tau'}\|_{*}d\tau'}\Big(N^{\frac{\sf+2}{\d}-1}\indic_{\d-4\le \sf <\d-2} + N^{-\al_2}\indic_{\sf<\d-4}\Big) \\
+\Cs{[1-e^{-\frac{\sf\tau}{4}}]}N^{-\al_1}\sup_{[0,\tau]}\|\bmu^\tau\|_{L^\infty}^{\frac{\sf}{\d}}e^{-\ka\tau + C\int_0^\tau \|\bar{u}^{\tau'}\|_{*}d\tau'} .
\end{multline}
After relabeling of constants, this yields \eqref{eq:MFEgron21}. Reverting to the original $(t,x)$ coordinates, we arrive at \eqref{eq:MFEgron21'}.


We now consider the case $\M=-\I$ and $\sf\in [0,\d)$. Let $C_{LS,N}^\tau$ denote the LSI constant of ${\Q}_{N,\be}^{\tau}(\bmu^\tau)$, which we allow to be infinite, and, recycling notation, set $\ka \coloneqq \inf_{\tau\ge 0}\frac{1}{C_{LS,N}^\tau\be}$. Following the calculations of \cite{RS2023lsi},
\begin{align}
-\frac{1}{\be^2}I_N(\bar{f}_N^\tau \vert {\Q}_{N,\be}^{\tau}(\bmu^\tau)) &\le -{\frac{\ka}{\be}}H_N(\bar{f}_N^\tau \vert {\Q}_{N,\be}^{\tau}(\bmu^\tau)) \nn\\
 &= -\ka\Big(\bar{E}_N^\tau(\bar{f}_N^\tau, \bmu^\tau)- \frac{\log \K_{N,\be}^\tau(\bmu^\tau)}{N}\Big) \nn\\
 &=-\ka\bar{\Ec}_N^\tau(\bar{f}_N^\tau,\bmu^\tau) + \ka\Big(\frac{\log \K_{N,\be}^\tau(\bmu^\tau)}{N}  - \bar{\os}_N^\tau\Big).
\end{align}
 Applying the preceding estimate to the right-hand side of \eqref{eq:MFEIpre}, then arguing as in the proof of \cref{prop:MFEgron1},
\begin{multline}
\frac{d}{d\tau}\bEc^\tau(\bar{f}_N^\tau, \bmu^\tau) \le \Big(-\ka + C\|\bar{u}^\tau\|_{*}\Big)\Ec_N^\tau(\bar{f}_N^\tau, \bmu^\tau)  + \ka\Big(\frac{\log \K_{N,\be}^\tau(\bmu^\tau)}{N}  - \bar{\os}_N^\tau\Big)\\
+\frac{1}{4N}\indic_{\substack{\sf=0 \\ }} +\frac{\sf e^{-\sf\tau/2}  \mathsf{C}\|\bmu^\tau\|_{L^\infty}^{\frac{\sf}{\d}}}{4}\begin{cases} N^{\frac{\sf}{\d}-1}, & {\sf\geq\d-2} \\ N^{-\al_1}, & {\sf<\d-2}.\end{cases}
\end{multline}
Recalling the definition \eqref{eq:bosNdef} of $\bar{\os}_N^\tau$, we see that
\begin{align}
-\bar{\os}_N^\tau \le \frac{\log_-(N\|\bmu^\tau\|_{L^\infty})}{2\d N}\indic_{\sf=0},
\end{align}
where $\log_{-}(\cdot) = -\min(\log(\cdot),0)$. As shown in \cite[Appendix A]{RS2023lsi},
\begin{align}
\frac{{\log \K_{N,\be}^\tau(\bmu^\tau)}}{N} \le \be e^{-\frac{\sf\tau}{2}}\begin{cases} \frac{\log(N\|\bmu^\tau\|_{L^\infty})}{2\d N}\indic_{\sf=0} + \Cs\|\bmu^\tau\|_{L^\infty}^{\frac{\sf}{\d}}N^{\frac{\sf}{\d}-1}, & {\d-2\le \sf<\d} \\ \frac{\Cs(\log  N + |\log\|\bmu^\tau\|_{L^\infty}|)}{N}\indic_{\sf=0} + \Cs\|\bmu^\tau\|_{L^\infty}^{\frac{\sf}{\d}}N^{-\al_1}, & {\sf<\d-2}. \end{cases}
\end{align}
Simplifying, we arrive at
\begin{multline}
\bar{\Ec}_N^\tau(\bar{f}_N^\tau, \bmu^\tau) \le e^{C\int_0^\tau \|\bar{u}^{\tau'}\|_{*}d\tau'}\Bigg[e^{-{\ka}\tau }\Ec_N^0(\bar{f}_N^0, \bmu^0) + \frac{[1-e^{-\ka\tau}]}{4N}\Big[1+\frac{2}{\d}\sup_{[0,\tau]}\log_-(N\|\bmu^{\tau'}\|_{L^\infty})\Big]\indic_{\sf=0}\\
+ \frac{\ka\be e^{-\ka\tau}[e^{(\ka-\frac{\sf}{2})\tau}-1]}{\ka-\frac{\sf}{2}}  \sup_{[0,\tau]}\begin{cases} \frac{\log(N\|\bmu^{\tau'}\|_{L^\infty})}{2\d N}\indic_{\sf=0} + \Cs\|\bmu^{\tau'}\|_{L^\infty}^{\frac{\sf}{\d}}N^{\frac{\sf}{\d}-1}, & {\d-2\le \sf<\d} \\ \frac{\Cs(\log  N + |\log\|\bmu^{\tau'}\|_{L^\infty}|)}{N}\indic_{\sf=0} + \Cs\|\bmu^{\tau'}\|_{L^\infty}^{\frac{\sf}{\d}}N^{-\al_1}, & {\sf<\d-2} \end{cases}\\
+\frac{\Cs\sf e^{-\ka\tau}[e^{(\ka-\frac{\sf}{2})\tau}-1]}{4(\ka-\frac{\sf}{2})}\Big(N^{\frac{\sf}{\d}-1}\indic_{\sf\ge \d-2} + N^{-\al_1}\indic_{\sf<\d-2} \Big)\sup_{[0,\tau]}\|\bmu^{\tau'}\|_{L^\infty}^{\frac{\sf}{\d}}\Bigg].
\end{multline}
After relabeling constants, this gives \eqref{eq:MFEgron21'}. Reverting to $(t,x)$ coordinates and relabeling the constants, we obtain \eqref{eq:MFEgron22'}.
\end{proof}

\begin{remark}\label{rem:MFEhamLSItorus}
We elaborate on using the modulated free energy to obtain generation of chaos on the torus for the Hamiltonian case $\sf \in (0,\d-2)$ mentioned in \cref{rem:MFEham}. Since the self-similar transformation is not relevant for the torus, there is no $-\frac{\sf}{4}\E_{\bar{f}_N^\tau}[\bar{\Fr}_N(\Xi_N,\bmu^\tau)]$ term in the periodic case. However, as shown in \cite{RS2021},\footnote{Strictly speaking, this work only considers the whole space; but the argument is adaptable to the torus using the decay estimates of \cite{CdCRS2023}.} one has a uniform-in-time bound for the expected modulated energy solely in terms of the initial modulated energy (i.e., there is no need for the relative entropy). Hence, a caricatured differential inequality
\begin{align}
\frac{d}{d\tau}\Ec_N^\tau(\bar{f}_N^\tau,\bmu^\tau) &\le -\frac1\be e^{-\|\log\frac{\bmu^\tau}{\bmu_\be}\|_{L^\infty}}H_N(\bar{f}_N^\tau \vert (\bmu^\tau)^{\otimes N}) + \bar{\os}_N^\tau(1)  \nn\\
&\le  -e^{-\|\log\frac{\bmu^\tau}{\bmu_\be}\|_{L^\infty}}\Ec_N^\tau(\bar{f}_N^\tau,\bmu^\tau) + \frac{1}{\be}e^{-\|\log\frac{\bmu^\tau}{\bmu_\be}\|_{L^\infty}}\Big(\E_{\bar{f}_N^\tau}[\Fr_N^\tau(\Xi_N,\bmu^\tau)] +\bar{\os}_N^\tau(1)\Big)
\end{align}
may be integrated, setting $\ka^\tau \coloneqq e^{-\|\log\frac{\bmu^\tau}{\bmu_\be}\|_{L^\infty}}$, to obtain
\begin{align}
\Ec_N^\tau(\bar{f}_N^\tau,\bmu^\tau) &\le e^{-\int_0^\tau \ka^{\tau'}d\tau'}\Ec_N^0(\bar{f}_N^0,\bmu^0) + \frac1\be\int_0^\tau \ka^{\tau'} e^{-\int_{\tau'}^\tau \ka^{\tau''}d\tau''}\Big(\E_{\bar{f}_N^{\tau'}}[\Fr_N^{\tau'}(\Xi_N,\bmu^\tau)] +\bar{\os}_N^{\tau'}(1)\Big)d\tau'.
\end{align}
Noting that
\begin{multline}
\int_0^\tau \ka^{\tau'} e^{-\int_{\tau'}^{\tau}\ka^{\tau''}d\tau''}\Big(\E_{\bar{f}_N^{\tau'}}[\Fr_N^{\tau'}(\Xi_N,\bmu^\tau)] +\bar{\os}_N^{\tau'}(1)\Big) d\tau'\\
 \le \sup_{[0,\tau]}\Big(\E_{\bar{f}_N^{\tau'}}[\Fr_N^{\tau'}(\Xi_N,\bmu^\tau)] +\bar{\os}_N^{\tau'}(1)\Big) \Big[\ka^\tau - \ka^{0}e^{-\int_0^\tau \ka^{\tau'}d\tau'}\Big]
\end{multline}
completes the argument.
\end{remark}


\bibliographystyle{alpha}
\bibliography{../MASTER}

\end{document}